\documentclass{amsart}
\usepackage{amsmath,amssymb,amsthm,amsfonts,epsfig,verbatim}
\input{xy}                                                                      
\xyoption{all}
\CompileMatrices 

\renewcommand{\comment}[1]{}

\newcommand{\us}[1]{{\upshape{#1}}}
\newcommand{\bp}{\begin{pmatrix}}
\newcommand{\ep}{\end{pmatrix}}
\newcommand{\be}{\begin{equation}}
\newcommand{\ee}{\end{equation}}
\newcommand{\bs}{\begin{split}}
\newcommand{\es}{\end{split}}
\newcommand{\bc}{\begin{center}}
\newcommand{\ec}{\end{center}}
\newcommand{\ed}{{\rm d}}
\newcommand{\w}{{\mathchoice{\,{\scriptstyle\wedge}\,}{{\scriptstyle\wedge}}    
      {{\scriptscriptstyle\wedge}}{{\scriptscriptstyle\wedge}}}}                
\newcommand{\lhk}{\mathbin{\hbox{\vrule height1.4pt width4pt depth-1pt          
             \vrule height4pt width0.4pt depth-1pt}}}  
\newcommand{\ol}{\overline}

\renewcommand{\Re}{\operatorname{Re}}
\renewcommand{\Im}{\operatorname{Im}}
\newcommand{\mo}{\sqrt{-1}}
\newcommand{\om}{\omega}
\newcommand{\Om}{\Omega}

\newcommand{\I}[1]{{\rm I}^{(#1)}}

\newcommand{\R}{\mathbb R}
\newcommand{\C}{\mathbb C}

\newcommand{\Z}{\mathbb Z}

\newcommand{\s}[1]{{\mathbb S}^{#1}}

\newcommand{\mcc}{\mathcal C}

\newcommand{\mce}{\mathcal E}

\newcommand{\mch}{\mathcal H}
\newcommand{\mci}{\mathcal I}

\newcommand{\mcl}{\mathcal L}

\newcommand{\mcs}{\mathcal S}

\newcommand{\vp}{\varphi}

\newcommand{\del}{{\partial}}
\newcommand{\delb}{\bar{\partial}}
\newcommand{\ddv}[1]{\frac{\partial}{\partial {#1}}}
\newcommand{\dd}[2]{\frac{\partial {#1}}{\partial {#2}}}

\newcommand{\M}[1]{M^{(#1)}}
\newcommand{\mcip}[1]{{\mci}^{(#1)}}

\newcommand{\wt}{{\rm {wd}}}
\newcommand{\ob}[2]{\ol \Omega^{{#1},{#2}}}
\newcommand{\obf}[3]{\ol \Omega_{#3}^{{#1},{#2}}}
\newcommand{\et}[3]{E_{#3}^{{#1},{#2}}}

\newcommand{\ub}{\ol{u}}
\newcommand{\zb}{\ol{z}}
\newcommand{\zetab}{\ol{\zeta}}
\newcommand{\etab}{\ol{\eta}}
\newcommand{\omb}{\ol{\om}}

\newcommand{\hb}{{\bar H^1}}
\usepackage[mathscr]{eucal}
\usepackage{bbm}
\usepackage{oldgerm}

\theoremstyle{plain}
\newtheorem{thm}{Theorem}[section]
\newtheorem{lem}[thm]{Lemma}

\newtheorem{cor}[thm]{Corollary}

\newtheorem{prop}[thm]{Proposition}

\newtheorem{rem}[thm]{Remark}
\theoremstyle{definition}
\newtheorem{defn}{Definition}[section]

\newtheorem{exam}[thm]{Example}

\begin{document}
\title[higher-order conservation laws]{higher-order conservation laws for the non-linear Poisson equation via characteristic cohomology}
\author{Daniel Fox}
\address{24-29 St Giles', Mathematics Institute, University of Oxford, Oxford, OX1 3LB, UK}
\email{foxd@maths.ox.ac.uk}
\author{Oliver Goertsches}
\address{Mathematisches Institut, Universit\"at zu K\"oln, Weyertal 86-90, 50931 K\"oln, Germany}
\email{ogoertsc@math.uni-koeln.de}
\date{\today}
\begin{abstract}
We study higher-order conservation laws of the non-linearizable elliptic Poisson equation
$
\frac{{\partial}^2 u}{\partial z \partial \zb} = -f(u)
$
as elements of the characteristic cohomology of the associated exterior differential system.  The theory of characteristic cohomology determines a normal form for differentiated conservation laws by realizing them as elements of the kernel of a linear differential operator.  The $\s{1}$--symmetry of the PDE leads to a normal form for the undifferentiated conservation law as well. 

We show that for higher-order conservation laws to exist, it is necessary that $f$ satisfies a linear second order ODE.  In this case, an at most real two--dimensional space of new conservation laws in normal form appears at each even prolongation.  When $f_{uu}=\beta f$ this upper bound is attained and the work of Pinkall and Sterling \cite{Pinkall1989} allows them to be written explicitly.

We relate higher-order conservation laws to generalized symmetries of the exterior differential system by identifying their generating functions.   This Noether correspondence provides the connection between conservation laws and the canonical Jacobi fields of Pinkall and Sterling.  
\end{abstract}
\maketitle

\tableofcontents



%
%
%

\section{Introduction}
A select set of elliptic Poisson equations 
\be\label{eq:fGordon}
\frac{{\partial}^2 u}{\partial z \partial \zb}  = -f(u),
\ee
where $u:\C \to \R$, is central in the study of submanifold geometries:  When  
\be\label{eq:Potentials}
f(u)=
\begin{cases}
-\frac{(\epsilon+\delta^2)}{4}\sinh(2u) \;\;\;&{\rm if}\;\; \epsilon+\delta^2>0\\
-\frac{(\epsilon+\delta^2)}{4}\cosh(2u) \;\;\;&{\rm if}\;\; \epsilon+\delta^2<0\\
e^{-2u} \;\;\;\;\;&{\rm if}\;\; \epsilon+\delta^2=0 
\end{cases}
\ee
then Equation \eqref{eq:fGordon} arises as the Gauss equation for a surface of constant mean curvature $-2\delta$ in a three dimensional space form of constant sectional curvature $\epsilon$.  When 
\[f(u)=e^{-2u}-e^{u},\] 
Equation \eqref{eq:fGordon} is the Gauss equation for a special Legendrian surface in $\s{5}$.  In all of these cases, the metric on the surface is locally given by $e^{2u}\ed z \circ \ed \zb$.  Once one has a solution $u(z,\zb)$ of \eqref{eq:fGordon}, the map of the surface into the space form can be recovered by solving a system of ODE.   It is also well known that the hyperbolic equation $u_{xt}=\sin(u)$, where $x,t$ are coordinates on $\R^2$, is the Gauss equation for surfaces in $\R^3$ with Gauss curvature equal to $-1$.   

For all of the potentials $f(u)$ listed above, Equation \eqref{eq:fGordon} is often referred to as a \emph{soliton  equation} or \emph{integrable system} and is known to have many special properties, including a loop group formulation, infinitely many conserved quantities, and a description of solutions using algebraic geometry (the spectral curve).  The literature on soliton equations is fascinating but also sprawling and tangled. 

Integrable systems of the form \eqref{eq:fGordon} underlie the simplest cases of primitive maps from Riemann surfaces into $k$-symmetric spaces \cite{Burstall1995}.  For this perspective the reader might consult the articles by Uhlenbeck \cite{Uhlenbeck1989}, Pinkall and Sterling \cite{Pinkall1989}, Hitchin \cite{Hitchin1990}, Bobenko \cite{Bobenko1991}, Burstall \cite{Burstall1992}, Bolton et al.~\cite{Bolton1995}, Dai and Terng \cite{Dai2000}, McIntosh \cite{McIntosh2001}, and the references within.   Hyperbolic equations of the form $u_{tx}=f(u)$ fit into the hierarchies developed by Terng and Uhlenbeck \cite{Terng1980,Terng2000}.  All of these references use the fact that soliton equations can be phrased in terms of flat connections on a Riemann surface.  

A markedly different approach using recursion operators was initiated by Lenard (a private communication cited in \cite{Gardner1974}) and Olver \cite{Olver1977,Olver1993}, and later developed and formalized by, for example, Guthrie \cite{Guthrie1994}, Dorfman \cite{Dorfman1993}, and Sanders and Wang \cite{Sanders2001}.   

Yet another approach to investigating integrable systems is through the theory of characteristic cohomology developed by Bryant and Griffiths \cite{Bryant1995}.  They cite Vinogradov (see the references in \cite{Bryant1995}) as their main influence, but the theory of characteristic cohomology, and in particular its formulation in the special case of Euler--Lagrange systems, is also closely related to the work of Shadwick using the Hamilton--Cartan formalism  (\cite{Shadwick1980} and the references within) and the work of Olver \cite{Olver1993}. 

Thus far, the theory of characteristic cohomology has mostly been used as a method for classifying partial differential equations, or more generally, exterior differential systems (EDS).  In \cite{Bryant1995,Bryant1995a,Bryant1995b,Bryant1995c,Clelland1997,Clelland1997a,Wang2004} scalar parabolic and hyperbolic PDE for 2 and 3 independent variables are classified  (using the method of equivalence and the characteristic cohomology) in terms of the dimension of the space of conservation laws.  Bryant, Griffiths, and Hsu  \cite{Bryant1995,Bryant1995b,Bryant1995c,Bryant1995d} make many interesting suggestions for other ways in which it might be used, including, for example, to study boundaries of integral manifolds and to study singularities.  Motivated by this the first author introduced an elementary approach to studying boundaries of integral manifolds using conservation laws in \cite{Fox2009}. 

The references to the literature given above are by no means exhaustive or even representative.  They were highlighted to give examples of other approaches that turn out to have close links to the theory of characteristic cohomology.   It is not clear, for example, how the existence of hydrodynamic reductions (see \cite{Ferapontov2006} and the references within) relate to the existence of conservation laws.  No doubt there are many more approaches and many more connections to be made between the various techniques in the literature. 

In \cite{Bryant2003} (Proposition 4.6) it was shown that the non-linear Poisson equation $\Delta u = f(u)$, where $u: \R^n \to \R$ and $\Delta$ is the Laplacian, admits no non-classical conservation laws if $n \geq 3$ and $f_{uu} \neq 0$.  On the other hand, the class of equations $\Delta u = f(u)$ with $n=2$ and $u$ and $f(u)$ vector valued encompass the Toda equations, which are known to be integrable \cite{Bolton1995}.  It does not appear to be known if higher-order conservation laws exist for the non-linear Poisson equation when $n \geq 3$ and $u$ and $f(u)$ are vector valued. 

In this article we study the (possibly infinite dimensional) space of conservation laws of equations of the form \eqref{eq:fGordon} using the characteristic cohomology.   We show that for there to exist higher-order conservation laws, it is necessary that $f$ satisfies a linear second order ODE.  We give a complete and explicit description of the conservation laws in terms of the characteristic cohomology in the case that $f_{uu}=\beta f$ and $f$ doesn't satisfy a linear ODE. We find that in this case the conservation laws for \eqref{eq:fGordon} are equivalent to those studied by Olver in the hyperbolic case \cite{Olver1977}, though his characterization is not complete because he doesn't prove that the necessary recursion operator is always well defined, nor does he prove that the method would produce the complete set of conservation laws.   There is also some overlap with the work of Dodd and Bullough \cite{Dodd1977}, though they also do not address the issue of completeness. 

The conservation laws turn out to be equivalent to the canonical Jacobi fields of \cite{Pinkall1989} and thus to the formal Killing fields of \cite{Burstall1993}.  This is not surprising given the Noether correspondence between generalizaed symmetries and conservation laws (see Section \ref{sec:Symmetries}).    In a future article we will describe how the characteristic cohomology can be used to recapture the notion of finite type solutions \cite{Pinkall1989,Burstall1993}.  We also hope to elaborate on the relationship between conservation laws and formal/polynomial Killing fields for primitive map systems and to study the structure of the characteristic cohomology on solutions with compact domains.

We conclude this section with a sketch of the remainder of the article.   In Section \ref{sec:EDS} we reformulate \eqref{eq:fGordon} as an exterior differential system \eqref{eq:EDS} and present the structure equations for the $k^{th}$ prolongation $(\M{k},\mci^{(k)})$, allowing $k=\infty$.   An $\s{1}$--symmetry of the PDE leads to an $\s{1}$--symmetry of  $(\M{k},\mci^{(k)})$ and to the notion of weighted degree for functions and differential forms.  We also introduce an almost complex structure $J$ on a codimension--$1$ subbundle of $T^*\M{k}$ which leads to $\del$ and $\delb$ operators.

In Section \ref{sec:CharacteristicCohomology} we present the basic definition of classical and higher-order conservation laws for an EDS, in both their differentiated and undifferentiated forms.  The classical conservation laws for \eqref{eq:EDS} are presented in Section \ref{sec:Classical}.  In Section \ref{sec:FirstApproximation} we use the general theory \cite{Bryant1995} to obtain the first approximation to the (differentiated) conservation laws.  In Section \ref{sec:AlgebraicForm} we refine the first approximation, obtaining an exact formula for differentiated conservation laws in normal form in terms of a generating function which is a solution to an (overdetermined) system of linear PDE, Equations \eqref{eq:NoMixing} and \eqref{eq:ED}.  Complicated calculations that would be necessary to directly verify that this formula does in fact convert solutions of \eqref{eq:NoMixing} and \eqref{eq:ED} into conservation laws (i.e.~to show that the thus defined differential forms are closed) are circumvented by studying (weighted) homogeneous conservation laws in Section \ref{sec:Homogeneous}.  The $\s{1}$--symmetry of the EDS allows one to produce  from the differentiated conservation laws  a normal form for undifferentiated conservation laws -- something that has not appeared in the general theory but is likely to be generally applicable to systems that have a gauge symmetry.\footnote{See Section \ref{sec:Conclusion}.}

In Section \ref{sec:Existence} we use the normal form of undifferentiated conservation laws to show that any solution to  \eqref{eq:NoMixing} and \eqref{eq:ED} defines a nontrivial conservation law.  Furthermore, we show that there is an at most one--dimensional complex space of solutions of  \eqref{eq:NoMixing} and \eqref{eq:ED} for each odd weighted degree, none of nonzero even weighted degree, and that these solutions are either `holomorphic' or `anti-holomorphic' polynomials in the derivatives $\frac{\partial^i u}{\partial z^i}$.  

In Section \ref{sec:fuu=betaf} we investigate the space of solutions of  \eqref{eq:NoMixing} and \eqref{eq:ED} under certain assumptions on $f$. We prove that if $f$ does not satisfy a linear second order ODE, no higher-order conservation laws exist. When $f_{uu}=\beta f$ and $f$ does not satisfy any first--order ODE, we use the work of Pinkall and Sterling \cite{Pinkall1989} to produce the complete set of generating functions, and hence the complete (infinite dimensional) space of conservation laws. We also provide examples of higher-order conservation laws for the case when $f_{uu}=\alpha f_u + 2 \alpha^2 f$.  In this case a coordinate change of \eqref{eq:fGordon} transforms it to the Tzitzeica equation $ u _{z \zb}=e^u- e^{-2u}$.

In Section \ref{sec:Symmetries} we show that generalized symmetries of $(\M{\infty},\mci^{(\infty)})$ are determined by generating functions that are solutions to \eqref{eq:ED}, though they need not satisfy \eqref{eq:NoMixing}.  This leads to a limited version of Noether's theorem which explains the relationship between conservation laws and the canonical Jacobi fields of Pinkall and Sterling \cite{Pinkall1989}. Section \ref{sec:Conclusion} contains some concluding remarks.\\

{\it Acknowledgments:} The authors would like to thank Dominic Joyce for many helpful conversations and Jenya Ferapontov for pointing out the reference \cite{Ziber1979}.  The work on this article began when they were at UC Irvine and the second author was supported by a DAAD postdoctoral scholarship. They wish to thank UC Irvine and Chuu-Lian Terng for their hospitality.  While finishing this work the first author was at Oxford University, supported by National Science Foundation grant OISE-0502241.  He thanks Oxford University and Dominic Joyce for their hospitality.

\section{The EDS and its prolongations}\label{sec:EDS}
To begin, we encode the PDE as an exterior differential system (EDS) with independence condition.\footnote{For a basic introduction to EDS see \cite{Bryant1991} or \cite{Ivey2003}.} Recall that an {\bf exterior differential system} consists of a smooth manifold $M$ and a homogeneous differential ideal $\mci \subset\bigoplus_p  \Om^p(M,\C)$.  An {\bf integral manifold} of $(M,\mci)$ is an immersed submanifold $\iota:N \to M$ such that $\iota^*(\mci)=0$.  If the ideal is generated by forms $\alpha_i $ (and their exterior derivatives since it is a \emph{differential} ideal) we will write $\mci= \langle\alpha_i\rangle$.  For any set of $1$-forms $\beta_i\in \Om^1(M,\C)$, we use $\{\beta_i\} \subset \Om^1(M,\C)$ to denote the subbundle they span.  If $\alpha \in \mci$ is a complex valued differential form then by $\iota^*(\alpha)=0$ we mean that both the real and imaginary parts pull back to the real manifold $N$ to be zero.  

In order to encode \eqref{eq:fGordon} as an EDS, let $M=\C^2 \times \R$ have coordinates $(z,u_0,u)$ and define the differential forms
\begin{align*}
\zeta&=\ed z\\ 
\om_1&=\ed u_0 +f \zetab\\
\eta_0&=\ed u - u_0 \zeta - \ub_0 \zetab\\
\psi&=\Im(\zeta \w \om_1)=-\frac{\mo}{2}(\zeta \w \om_1-\zetab \w \omb_1).
\end{align*}
The reader may recognize $M$ as the first jet space of maps $u:\C \to \R$. The desired differential ideal is $\mci=\left\langle \eta_0, \psi  \right\rangle $.  We calculate that 
\begin{align*}
\ed \eta_0 &=2\Re(\zeta \w \om_1)=\zeta \w \om_1 + \zetab \w \omb_1 \\
\ed \psi &=-{\mo} f_u \eta_0 \w \zeta \w \zetab.
\end{align*}
Thus the differential ideal can be expressed as 
\[
\mci=\langle \eta_0, \zeta \w \om_1 \rangle.
\]
One checks that a surface $\iota:\C \to M$ for which $\iota^*(\zeta \w \zetab) \neq 0$ and $\iota^*(\eta_0)=0$ is the $1$--jet of a function $u:\C \to \R$ with $u_0=\dd{u}{z}$ and $\ub_0=\dd{u}{\zb}$.  If in addition $\iota^*(\psi)=0$ then the function $u(z,\zb)$ is a solution to \eqref{eq:fGordon}. Thus solutions to \eqref{eq:fGordon} correspond to integral surfaces $(N,\iota)$ such that $\iota^*(\zeta \w \zetab ) \neq 0$. 
 
The goal of this article is to study the conservation laws of the EDS
\be\label{eq:EDS}
(M,\mci),\;\;{\rm where}\;\; M=\C^2 \times \R\;\; {\rm and}\;\; \mci=\langle \eta_0, \psi \rangle,
\ee
\emph{and} its prolongations.   This EDS is involutive with Cartan characters\footnote{Again, see \cite{Bryant1991,Ivey2003} for the basics of EDS.} $s_0=1,s_1=2,s_2=0$.  

We outline the process of the first prolongation.  If $\iota:N \to M$ is an integral manifold then its tangent space at $\iota(n)$ is a real $2$-plane $\iota_*(T_n N) \subset T_{\iota(n)}M$ on which the ideal pulls back to be zero.  Any real $2$-plane $E \subset T_{\iota(n)}M$ on which $\zeta \w \zetab \neq 0$ is defined by relations\footnote{We are using the notation $E^*(\eta_0)$ to indicate $\eta_0$ pulled back to the plane $E$.}
\begin{align*}
E^*(\eta_0)&=c_1 \zeta +\ol{c}_1 \zetab\\
E^*(\om_1)&=c_2 \zeta + c_3 \zetab
\end{align*}
for some complex numbers $c_1,c_2,c_3$.  The ideal $\mci=\langle \eta_0,\zeta \w \om_1 \rangle$ vanishes on $E$ if and only if $c_1=c_3=0$.  Thus the space of possible tangent planes to integral manifolds is parametrized by one complex number, which we will call $u_1$, via the conditions $E^*\eta_0=0$ and $E^*\om_1=u_1 \zeta$.  

Let $\M{1}=M \times \C$ and let $u_1$ be a holomorphic coordinate on $\C$.  Define the complex one form $\eta_1=\om_1-u_1 \zeta$ and the subbundle $\I{1}=\{ \eta_0, \eta_1, \etab_1 \} \subset \Om^1(\M{1},\C)$ which generates a differential ideal $\mcip{1}$.  The new system $(\M{1},\I{1})$ is the first prolongation of $(M,\mci)$ with respect to the independence condition $\zeta \w \zetab \neq 0$.   Thus we construct the prolongation by adjoining a new coordinate parametrizing the possible tangent spaces to integral manifolds and introducing tautological $1$--forms that vanish on potential tangent planes to integral manifolds. 

So what is the meaning of $u_1$?  It contains the new second order information of $u(z,\zb)$:  The vanishing of $\eta_0=\ed u - u_0 \zeta - \ub_0 \zetab$ implies that $u_0=\dd{u}{z}$.  The vanishing of $\eta_1=\ed u_0 -u_1 \zeta +f \zetab$ implies that $u_1=\dd{u_0}{z}$ and $-f=\dd{u_0}{\zb}$.  The first tells us that $u_1=\frac{\partial^{2}u}{\partial z^{2}}$ -- the new second derivative information on $u$ -- and the second of these recaptures the PDE condition that was encoded in the vanishing of $\psi$.  
\comment{ Any two plane $E \subset T_mM$ on which $ \zeta \w\ \zetab \neq 0 $ and on which the ideal $\mci$ pulls back to be zero is defined by the conditions
\begin{align*}
E^*(\eta_0)&=0\\
E^*(\eta_1)&=0
\end{align*}
for some complex number $u_1$.}

Now using the fact that $\ed \eta_1$ must vanish on solutions of $(\M{1},\mci^{(1)})$, one can find the possible tangent planes of solutions of $(\M{1},\mci^{(1)})$ and in the same way as before construct the second prolongation.  Let $\M{k}$ denote the $k^{th}$-prolongation.  It is not hard to see that $\M{k+1}=\M{k} \times \C$ and we will always use $u_{k+1}$ for the new holomorphic coordinate on $\M{k+1}$.  Furthermore, on a (real) two-dimensional integral manifold $\iota:N \to M$ for which $\iota^*(\zeta \w \zetab) \neq 0$, 
\[
\iota^*(u_i)=\frac{\partial^{i+1}u}{\partial z^{i+1}}.
\]

By calculating the first few prolongations one is motivated to define complex functions and forms
\bc
\begin{tabular}{ll}
&$\om_{i+1}=\ed u_{i}+T^{i}\zetab$\\
$T^0=f$&$\eta_0=\ed u -u_0\zeta-\ub_0\zetab$ \\
$T^{i+1}=\sum_{j=0}^{i}{i \choose j}u_{i-j}T^j_u $&$ \eta_{i+1}=\om_{i+1}-u_{i+1} \zeta$\\
&$\tau^i=\sum_{j=0}^{i}{i \choose j} T^j_u \eta_{i-j}$.\\
\end{tabular}
\ec
The real and imaginary parts of $\zeta,\eta_0, \dots,\eta_{k},\om_{k+1}$ form a coframe of $\M{k}$ and
\[
\I{k}=\{\eta_0,\eta_{1},\etab_{1},\ldots ,\eta_{k},\etab_{k} \} \subset \Om^1(\M{k},\C)
\]
generates the ideal $\mci^{(k)}$. The vector fields on $\M{k}$ dual to this coframe are
\bc
\begin{tabular}{lll}
$e^k_{-1}=\ddv{z}+u_0\ddv{u}+\sum_{i=0}^{k-1} u_{i+1} \ddv{u_{i}}-\sum_{i=0}^{k}\ol{T}^{i}\ddv{\ub_{i}}$ & & $\zeta$ \\
$e_0= \ddv{u}$& $\longleftrightarrow$& $\eta_0$\\
$e_i=\ddv{u_{i-1}} \;\;\; i=1 \ldots k+1$& &$\eta_i \;\; (i=1 \ldots k),\;\;\;\om_{k+1}$\\
\end{tabular}
\ec
and their complex conjugates.\footnote{Although the notation $e^k_0$ and $e^k_i$ would be more correct because, for example, $e^k_0$ and $e^{k'}_0$ are vector fields on different manifolds, we drop the indexing of the prolongation because the same formula holds and there are natural inclusions and surjections between $\M{k}$ and $\M{k'}$ which identify the corresponding vector fields.  We leave the superscript on $e^k_{-1}$ because this vector field does change from prolongation to prolongation.}

To compute the structure equations we will need
\begin{lem}\label{lem:tidentities} For $i,j\geq 0$ we have
\begin{enumerate}
\item $T^i=(e^k_{-1})^i f$ for $k\geq i$
\item $T^{i+j+1}_{u_j}={i+j+1 \choose i} T^i_u$.
\end{enumerate}
\end{lem}
\begin{proof}
We use induction and the binomial identity  ${i-1 \choose j}+{i-1 \choose j-1}={i \choose j}$ for both formulas.  We illustrate the calculation for the second formula only, making use of the first identity in the calculation.  Suppose that the second formula holds for all $n<i+1$ and all $j$.  Then 
\begin{align*}
T^{i+j+1}_{u_j}&=e_{j+1} e^k_{-1}T^{i+j}=(e^k_{-1} e_{j+1}+e_j) T^{i+j}\\
&= {i+j \choose i-1}e^k_{-1}T^{i-1}_u +{i+j \choose i}T^i_u ={i+j+1 \choose i} T^i_u.
\end{align*}
In the last equality we used the fact that $[e^k_{-1},e_0] T^{i-1}=0$ because $[e^k_{-1},e_0]$ is in the span of the $e_{\ol{i}}$, which annihilate $T^r$ for all $r$. 
\end{proof}
Thus on an integral manifold $T^i=\frac{\partial^i f}{{\partial z}^i}$.  Using Lemma \ref{lem:tidentities} it is not hard to compute that
\be\label{eq:Commutator}
[e^k_{-1},e^k_{\ol{-1}}]=\ol{T}^{k+1} \ol{e}_{k+1}-{T}^{k+1} {e}_{k+1},
\ee
which will be needed later.

\begin{prop}
For $1\leq i\leq k$ the following structure equations are satisfied on $\M{k}$:
\begin{align*}
\ed T^i &\equiv  T^{i+1}\zeta+ \tau^i  \;\;\; {\rm mod}\; \zetab\\
\ed \zeta&=0\\
\ed \eta_0&=\zeta \w \eta_1 + \zetab \w \etab_1\\
\ed \om_{k+1}&=\tau^{k} \w \zetab+T^{k+1} \zeta \w \zetab\\
\ed \eta_{i}&=- \eta_{i+1} \w \zeta+\tau^{i-1} \w \zetab
\end{align*}
\end{prop}
\begin{proof}
The second and the third equation follows directly from the definitions and the last two equations follow easily from the first.  To prove the first equation we calculate, using Lemma \ref{lem:tidentities} to differentiate $T^i$:
 \begin{align*}
\ed T^i \equiv  T^{i+1}\zeta+\sum_{j=0}^i T^i_{u_{j-1}} \eta_j \equiv T^{i+1}\zeta+\sum_{j=0}^i {i \choose j} T^{i-j}_u \eta_j \equiv T^{i+1}\zeta+ \tau^i,
\end{align*} 
where all equivalences are modulo $\zetab$.
\end{proof}

It will also be convenient to work on the infinite prolongation $\M{\infty}$ (see Section 4.3.1 of  \cite{Bryant2003}).  The infinite dimensional space $ \M{\infty}$ is the inverse limit of the sequence
\[\left\lbrace \ldots \to \M{k} \overset{\pi_{k}}{\to} \M{k-1}\overset{\pi_{k-1}}{\to} \ldots \to \M{1} \overset{\pi_{1}}{\to} \M{0} \right\rbrace;  \]
that is,
\[
\M{\infty}=\left\lbrace   (p_0,p_1,\ldots) \in \M{0} \times \M{1} \times \ldots : \pi_{k}(p_k)=p_{k-1}\;\; {\rm for \; each}\; k \geq 1  \right\rbrace.
\]
Let $\pi_{(k)}:\M{\infty} \to \M{k}$ be the natural surjections 
\[\pi_{(k)}(p_0,p_1,\ldots)=(p_0,p_1,\ldots,p_k).\]
A smooth function or differential form on $\M{\infty}$ is given by the pullback via $\pi_{(k)}$ of a corresponding object on some finite prolongation.     On $\M{\infty}$ the real and imaginary parts of $\zeta,\eta_0, \eta_1, \dots$ form a coframe and the dual vector fields on $\M{\infty}$ are the real and imaginary parts of 
\bc
\begin{tabular}{lll}
$e_{-1}=\ddv{z}+u_0\ddv{u}+\sum_{i=0}^{\infty} u_{i+1} \ddv{u_{i}}-\sum_{i=0}^{\infty}\ol{T}^{i}\ddv{\ub_{i}}$ & & $\zeta$ \\
$e_0= \ddv{u}$& $\longleftrightarrow$& $\eta_0$\\
$e_i=\ddv{u_{i-1}} \;\;\; i=1 \ldots k+1$& &$\eta_i \;\; (i=1,2,\ldots)$\\
\end{tabular}
\ec
The ideal is generated by the (formally Frobenius) subbundle
\[
\I{\infty}=\{\eta_0,\eta_{1},\etab_{1},\eta_2, \etab_2,\ldots \}.
\]
Note that if $F\in \Om^0(\M{k},\C)$ then $\pi_{(k)}^*(e_{-1}^k F)=e_{-1}( \pi_{(k)}^*(F))$.

Thus far we have calculated a coframe, its dual frame, and the structure equations of an arbitrary prolongation of $(M,\mci)$.  We now turn to some of the special structures on $(\M{k},\mci^{(k)})$ that arise due to the ellipticity and the $\s{1}$--symmetry. 

The PDE \eqref{eq:fGordon} is invariant under the $\s{1}$--action $(u,z,\zb) \to (u,\lambda z,\ol{\lambda} \zb)$ (with $\lambda \in \C$ and $|\lambda|=1$). This leads to a symmetry of $(\M{k},\mci^{(k)})$, which yields a decomposition of differential forms and thus conservation laws.   To see this, let $F:\s{1} \times \M{k} \to \M{k}$ be defined as 
\be\label{eq:Symmetry}
F(\lambda,u,z,u_j)=(u,\lambda^{-1} z,\lambda^{j+1}u_j).
\ee
For $p \geq 0$ and $j \in \Z$ we define the spaces of differential forms of {\bf homogeneous weighted degree} $j$ to be
\be\label{eq:wdegforms}
\Om^p_j(\M{k}) =\left\{\vp \in \Om^p(\M{k},\C) \mid F^*\vp=\lambda^j \vp \right\}.
\ee
For an element $\vp \in \Om^p_j(\M{k},\C)$, we write $\wt(\vp)=j$. Note that $\wt(\vp)=-\wt(\ol{\vp})$. This grading is preserved by exterior differentiation:  
\[
\ed: \Om^p_j(\M{k}) \to \Om^{p+1}_j(\M{k}).
\]
 One calculates that
\begin{center}
\begin{tabular}{ll}
$\wt(z)=-1$&$ \wt(\zb)=1$\\
$\wt(u_j)=+(j+1)$&$ \wt(\ub_j)=-(j+1)$\\
$\wt(u)=0$ & $\wt(\eta_0)=0$\\
$\wt(\zeta)=-1$ & $\wt(\zetab)=1$ \\
 $\wt(\om_j)=+j $&$ \wt(\omb_j)=-j$\\
 $\wt(\eta_j)=+j $&$ \wt(\etab_j)=-j$\\
\end{tabular}
\end{center}
We will use this symmetry in Sections \ref{sec:Homogeneous} and \ref{sec:Existence}.

The ellipticity of \eqref{eq:fGordon} leads us to the following
\begin{defn}
Define the subspaces $\Om^{(1,0)}(\M{k})=\C \cdot \{\zeta,\eta_1,\ldots,\eta_k,\om_{k+1} \}$ and $\Om^{(0,1)}(\M{k})=\C \cdot \{\zetab,\etab_1,\ldots,\etab_k,\omb_{k+1} \}$, and in the standard way also $\Om^{(p,q)}(\M{k})$. We define the operators $\del:C^{\infty}(\M{k},\C) \to \Om^{(1,0)}(\M{k})$ and $\delb:C^{\infty}(\M{k},\C) \to \Om^{(0,1)}(\M{k})$  as
\begin{align*}
\del A&=e_{-1}^k(A)\zeta+\sum_{i=1}^{k}A_{u_{i-1}} \eta_i + A_{u_k}\om_{k+1}\\
\delb A&=e^k_{\ol{-1}}(A)\zetab+\sum_{i=1}^{k}A_{u_{i-1}} \etab_i + A_{u_k}\omb_{k+1},
\end{align*}
allowing for $k=\infty$.
\end{defn}

It will be convenient to use the following linear operator 
\[
J:\Om^{1}(\M{k}) \to \Om^{1}(\M{k}),
\]
which acts by $\mo$ on $\Om^{(1,0)}(\M{k})$, by $-\mo$ on  $\Om^{(0,1)}(\M{k})$, and as the identity on $\R \cdot \eta_0$. This is an almost complex structure on the annihilator of $e_0$.

\section{Conservation laws as elements of the characteristic cohomology}\label{sec:CharacteristicCohomology}
Let   $(M,\mci)$ be an involutive exterior differential system with maximal integral submanifolds of dimension $n$ and characteristic number $l$.\footnote{The characteristic number is computed from $\mci$ using linear algebra.  See Section 4.2 of \cite{Bryant1995} for the definition.   It is often $1$, as it is for \eqref{eq:EDS}. }   Its characteristic cohomology is defined to be
\[
H^p_0(M,\Om/\mci),
\]
that is, the cohomology for the complex $\Om/\mci$ with differential $\ol{\ed}:{\Om^p}/({\mci \cap \Om^p}) \to {\Om^{p+1}}/{(\mci  \cap \Om^{p+1})} $ induced by the standard exterior derivative.   The subscript $0$ indicates that we are working on the zeroth prolongation.  We will say that we are in the {\bf local} case if $H^p_{dR}(M,\R)=0$ for $p>0$.   In \cite{Bryant1995} it is shown that in the local involutive case $H^p_0(M,\Om/\mci)=0$ for $p<n-l$.  The first nontrivial group is of special interest. 
\begin{defn}
The space of {\bf classical undifferentiated conservation laws} for $(M,\mci)$ is $H^{n-l}_0(M,\Om/\mci)$. 
\end{defn}
\begin{rem}
For the system \eqref{eq:EDS} we have $n=2$ and $l=1$ so that the space of classical undifferentiated conservation laws is $H^{1}_0(M,\Om/\mci)$.
\end{rem}
In addition to the quotient complex $(\Om/\mci,\ol{\ed})$ we also have the  subcomplex $(\mci \cap \Om^p(M,\R),\ed)$ and its cohomology $H_0^{p}(M,\mci)$.  To calculate the conservation laws in the local case one uses the isomorphism 
\[H_0^{n-l}(M,\Om/\mci) \cong H_0^{n-l+1}(M,\mci)\]
which follows from the long exact sequence in cohomology, which is induced by the short exact sequence 
\[
0 \to \mci \to \Om \to \Om/\mci \to 0
\]
and the fact that $H^p(M,\R)=0$  for $p>0$.   
\begin{defn}
The space of {\bf classical differentiated conservation laws} for $(M,\mci)$ is $H^{n-l+1}_0(M,\mci)$. 
\end{defn}
An element of $H_0^{n-l+1}(M,\mci)$ is a closed $(n-l+1)$--form in the ideal and we only care about it modulo $\ed$ of $(n-l)$--forms in the ideal.  The characteristic cohomology machinery developed in \cite{Bryant1995} identifies the space of conservation laws as the kernel of a linear differential operator (as opposed to elements of a quotient space), much as one finds harmonic representatives of de Rham classes in Hodge theory. 

On $(\M{\infty},\mci^{(\infty)})$ one has the associated characteristic cohomology, which we abbreviate as $\bar{H}^p:=H^p(\M{\infty},\Om/\mci)$.   We continue to restrict to the local case, $H^p(\M{k},\R)=0$ for $k \geq 0$ and $p>0$.
\begin{defn}
The space of {\bf higher-order undifferentiated conservation laws} for $(M,\mci)$ is 
\[
\bar{H}^{n-l}:=H^{n-l}(\M{\infty},\Om/\mci).\]
The space of {\bf higher-order undifferentiated {\it complex} conservation laws} for $(M,\mci)$ is 
\[\bar{H}^{n-l}_{\C}:=H^{n-l}(\M{\infty},\Om_\C/\mci_\C),\] where the subscript $\C$ denotes complexification. 
\end{defn}
Any element of $\bar{H}^{n-l}$ is represented by an element of $\Om^{n-l}(\M{\infty},\R)$ which, by definition, is the pullback under $\pi_{(k)}:\M{\infty} \to \M{k}$ of an element of  $\Om^{n-l}(\M{k},\R)$ for some $k$.  Again, we can study the conservation laws via the isomorphism $H^{n-l}(\M{k},\Om/\mci) \cong H^{n-l+1}(\M{k},\mcip{k})$ because   $H^p(\M{k},\R)=0$  for $p>0$.   
\begin{defn}
The space of {\bf higher-order differentiated conservation laws} for $(M,\mci)$ is $H^{n-l+1}(\M{\infty},\mci)$. The space of {\bf higher-order differentiated {\it complex} conservation laws} for $(M,\mci)$ is $H^{n-l+1}(\M{\infty},\mci_{\C})$.
\end{defn}
Exterior differentiation provides isomorphisms
\begin{align*}
 \ed:\bar{H}^{n-l}& \overset{\cong}{\to} H^{n-l+1}(\M{\infty},\mci)\\
 \ed:\bar{H}_{\C}^{n-l}& \overset{\cong}{\to} H^{n-l+1}(\M{\infty},\mci_{\C})
\end{align*} 
in the local involutive case.

\section{Classical conservation laws }\label{sec:Classical}
For the system \eqref{eq:EDS} the maximum integral manifolds are of dimension $2$ and the characteristic number is $1$.  In the notation of the last section, $n=2$ and $l=1$.  Thus a classical differentiated conservation law is represented by a closed form in $\mci \cap \Om^2(M,\R)$.  Any $2$--form in $\mci$ can be written as 
\[
\Phi'=\rho \w \eta_0 + A \psi + B \ed \eta_0
\]
for some $1$--form $\rho$ and functions $A$ and $B$.
This can be rewritten as
\[
\Phi'=(\rho -\ed B) \w \eta_0 + A \psi + \ed (B  \eta_0).
\]
As we are only interested in the class $[\Phi'] \in H^2(M,\mci)$, we need only concern ourselves with finding a $1$--form $\rho$ and a function $A$ such that 
\[
\Phi=\rho \w \eta_0 + A \psi 
\]
is closed and $\Phi \neq \ed \alpha$ for $\alpha \in \mci$.  It is easy to check that for any $\Phi$ of the form given, $\Phi \neq \ed (g \eta_0)$ for any function $g$. Examining $\eta_0 \w \ed \Phi=0$ uncovers that $\rho\equiv -\frac{1}{2}J\ed A\;\mod \eta_0$.  Then considering the terms in $\ed \Phi=0$ that involve $\eta_0$ uncovers the condition
\[
\frac{1}{2}\ed J \ed A-\mo f_u A \zeta \w \zetab + A_u \psi \equiv 0 \; \mod \; \eta_0.
\]
By studying the coefficients of this vanishing $2$--form one can check that if $\log(f)_{uu} \neq 0$ and $f_{uu} \neq 0$ then the only conservation laws are given by setting  
\begin{align*}
A&=P+\ol{P}\\
P&=a u_0+\mo bz u_0
\end{align*}
where $a \in \C$ and $b \in \R$ are arbitrary constants.   

In Section \ref{sec:Homogeneous} we will introduce a systematic way of finding undifferentiated conservation laws from differentiated ones.  It will fail only for the classical conservation laws with $a=0$ and $b \neq 0$.  For that reason we present here the $1$--form
\be\label{eq:wd0}
\vp_0=G \eta_0 +E \zeta+ \ol{E}\zetab
\ee 
with $G=-(zu_0+\zb \ub_0)$ and $E=-\frac{1}{2}z u_0^2+\zb \int f$. It satisfies $\ed \vp_0=\Phi$ when we take $P=\mo  z u_0$ in the definition of $\Phi$.  Notice that it is also an element of $\Om^1_0(M,\R)$, something we will make use of in Section \ref{sec:Existence}.

In order to look for higher-order conservation laws, that is, conservation laws of the prolonged system, we will make use of some spectral sequence machinery which we now describe.

\section{The first approximation of the characteristic cohomology}\label{sec:FirstApproximation}
The material in this section is based on Sections 1.3, 2.1--2.4, and 5.1 of \cite{Bryant1995}.  Let $(\M{\infty},\mci^{(\infty)})$ be the infinite prolongation of an involutive EDS $(M,\mci)$.   Assume that $H^p_{dR}(\M{k},\R)=0$ for all $p>0$ and for all $k \geq 0$ so that we are in the local involutive case.  The system \eqref{eq:EDS} is in the local involutive case.  Let $\I{\infty}=\mci_{\C}^{(\infty)} \cap \Om^1(\M{\infty},\C)$ and let $\Om^p=\Om^p(\M{\infty},\C)$.  Then define
\begin{align*}
F^p\Omega^{q}&=\Im\{\I{\infty} \w \ldots \w \I{\infty} \w \Omega^{q-p} \to \Omega^q   \}\\
\ob{p}{q}&=F^p\Omega^{q}/F^{p+1}\Omega^{q}.
\end{align*}
Let $(\et{p}{q}{r},{\ed}_r)$ denote the spectral sequence associated to this filtration \cite{Griffiths1994} whose first two terms are
\begin{align*}
\et{p}{q}{0}&=\ob{p}{q}\\
\et{p}{q}{1}&=H(\et{p}{q}{0},{\ed}_0)=\frac{\{\phi \in F^p\Omega^{q} : \ed \phi \in F^{p+1}\Omega^{q+1} \}}{\ed F^{p}\Omega^{q-1} \oplus F^{p+1}\Omega^{q}  }.
\end{align*}  
Notice that 
\[
{H}^q(\Om/\mci^{(\infty)},\ol{\ed})=\et{0}{q}{1}.
\]
We study this space indirectly using the spectral sequence $(\et{p}{q}{r},\ed_r)$, which includes the complex\\
\centerline{\xymatrix{0 \ar[r]&\et{0}{q}{1} \ar[r]^{\ed_1}& \et{1}{q+1}{1} \ar[r]^{\ed_1}& \et{2}{q+2}{1} \ldots}}
From Equation (4) of Section 4.2 in \cite{Bryant1995} we know that this sequence is exact at $ \et{0}{q}{1}$ and $ \et{1}{q+1}{1}$ so that the characteristic cohomology is isomorphic to
\[
\ker \{\ed_1:\et{1}{q+1}{1} \to \et{2}{q+2}{1}\}.
\]

A second spectral sequence allows one to obtain a first approximation to this kernel:  for this we use a weight filtration  
\[
\ol{F}_k\ol{\Om}^{p,q}=\{ \phi \in \ol{\Om}^{p,q}: wt(\phi) \leq k \} \subset \ob{p}{q},
\]
whose associated graded spaces we denote as $\obf{p}{q}{k}:=\ol{F}_k\ol{\Om}^{p,q}/\ol{F}_{k-1}\ol{\Om}^{p,q}$.  See Section 2.4 of \cite{Bryant1995} for the general definition of this weight filtration.  We will return to its properties and the weights for system \eqref{eq:EDS} in a moment.  First let's see how it will be used.  This quotient complex has the associated cohomology groups
\[
\mch_k^{p,q}=H^q(\obf{p}{*}{k},\delta),
\]
where $\delta$ is the differential induced by exterior differentiation.  Now for each fixed $p>0$ there is a spectral sequence $\{\bar{E}_r^{k,q} \}$ associated to the weight filtration that converges to $\et{p}{q}{1}$ and which satisfies
\[
\bar{E}_1^{k,q}=\mch_k^{p,q}.
\]
Thus computing $\mch_k^{p,q}$ gives us a first approximation of the form of a conservation law.  The importance of this step is that $\delta$ is linear over functions and computing $\mch_k^{p,q}$ is a purely algebraic process depending only on the principal symbol of the EDS.  

In order to compute the cohomology groups $\mch_k^{p,q}$ for \eqref{eq:EDS} we need the following properties of $wt$.  Let $f\neq 0$ be a smooth function and $\vp$ and $\psi$ smooth differential forms.  Then
\begin{align*}
wt(f \vp)&=wt(\vp) \\
wt(\vp \w \psi)&\leq wt(\vp)+wt(\psi)\\
wt(\vp + \psi)& \leq max(wt(\vp),wt(\psi)).
\end{align*}

Now suppose that $(\M{\infty},\mci^{(\infty)})$ is the infinite prolongation of \eqref{eq:EDS}.   Then we have the following weighting\footnote{This weighting system will not be used after this section and is distinct from the notion of weighted degree ($\wt{}$) defined in Section \ref{sec:EDS} and used throughout the paper.} system:
\begin{center}
\begin{tabular}{lll}
 $wt(\zeta)=wt(\zetab)=-1$ & $wt(\eta_0)=1$ & $wt(\eta_j)=wt(\etab_j)=j$\\
\end{tabular}
\end{center}
\comment{
The induced differential on the graded spaces $\obf{p}{q}{k}$ is complex linear and has structure equations
\begin{align*}
\delta \eta_0 &\equiv \zeta \w \eta_1+\zetab \w \etab_1 \in  \obf{p}{q}{0} \equiv 0 \in  \obf{p}{q}{1}\\
\delta \eta_j & \equiv \zeta \w \eta_{j+1} \in \obf{p}{q}{j}\\
\delta \zeta & \equiv 0.
\end{align*}

In the equations above and until specified otherwise, `$\equiv$' indicates equality among equivalence classes in $\obf{p}{q}{k}$.
}
To compute the necessary cohomology groups we need the spaces
\bc
\begin{tabular}{ll}
$\obf{1}{1}{-1}=0$&$\obf{1}{2}{-2}=0$\\
$\obf{1}{1}{0}=0$  &$\obf{1}{2}{-1}=0$ \\
$\obf{1}{1}{1}=\C \cdot \{\eta_0,\eta_1,\etab_1 \}$&$\obf{1}{2}{0}=\C \cdot \{\eta_0 \w \zeta, \eta_0 \w \zetab, \eta_1 \w \zeta, \eta_1 \w \zetab, \etab_1 \w \zeta, \etab_1 \w \zetab\}$ \\
$\obf{1}{1}{j}=\C \cdot \{\eta_j,\etab_j \}$ & $\obf{1}{2}{j}=\C \cdot \{\eta_{j+1} \w \zeta, \eta_{j+1} \w \zetab, \etab_{j+1} \w \zeta, \etab_{j+1} \w \zetab\} \;\;{\rm for}\;\; j>1$\\
\end{tabular}
\ec
Easy calculation uncovers that the only nonzero cohomology group for the complexes $(\obf{p}{q}{k},\delta)$ is
\[
\mch_0^{1,2}=\C \cdot \{\zeta \w \eta_1, \zetab \w \etab_1  \}.
\]
This implies that a conservation law can be represented by a form
\[
\tilde{\Phi} \equiv A'' \Re(\zeta \w \eta_1) + B'' \Im( \zeta \w \eta_1)  \;\;\; \mod \; F^2\Om^2.
\]
This can be rewritten as 
\[ \tilde{\Phi} \equiv A' \ed \eta_0 + B' \psi \equiv \eta_0 \w \ed A' + B' \psi+\ed(A'  \eta_0) \; \mod \; F^2\Om^2
\]
or, as we will continue on with, it can be written as 
\[
\Phi \equiv \eta_0 \w \rho+A\psi \;\; \mod \; F^2\Om^2+ \ed ( F^1 \Om^1)
\]
for some $1$--form $\rho$ and function $A$, where we have left off the exact piece because that does not alter its class in ${H}^2(\M{\infty},\mci^{(\infty)})$.  In the next section we remove the congruence, finding how the closure of $\Phi$ determines $\rho$ and the other coefficients in terms of $A$, as well as equations that $A$ must satisfy.  However, to prove the closure of $\Phi$ directly requires verifying some elaborate equations.  We circumvent this in Section \ref{sec:Existence} by using the normal form for undifferentiated conservation laws found in Section \ref{sec:Homogeneous}.

\section{The normal form of the  differentiated conservation laws}\label{sec:AlgebraicForm}
From now on assume that $f$ does not satisfy a first--order ODE.  We make the following definition.
\begin{defn}\label{def:normalform}
A representative $\Phi \in \mci^{(\infty)} \cap \Om^2(\M{\infty},\R)$ of a differentiated conservation law on $\M{\infty}$ is in {\bf normal form} if
\begin{align*}
\Phi&=\eta_0 \w \rho + A \psi   \\ 
&+\sum_{1\leq i < j \leq k } \left( B^{ij}\eta_i \w \eta_j +\ol{B}^{ij}\etab_i \w \etab_j\right) +\sum_{1\leq i \leq j \leq k } \left( D^{ij}\eta_i \w \etab_j +\ol{D}^{ij}\etab_i \w \eta_j \right).
\end{align*}
for some $k$ and some functions $A,B^{ij},D^{ij}:\M{\infty} \to \C$ with  $\ol{A}=A$. 
\end{defn}
There is an analogy between the role of conservation laws in normal form and harmonic representatives of de Rham cohomology classes which we now recall.  The de Rham cohomology of a smooth closed manifold $X$ consists of the \emph{quotient} groups $H^p_{dR}(X,\R)=\frac{\ker(\ed:\Om^{p}(X,\R) \to \Om^{p+1}(X,\R))}{\Im(\ed:\Om^{p-1}(X,\R) \to \Om^p(X,\R))}$.  Hodge theory shows that if one has a Riemannian metric, then one can represent these quotient spaces as subspaces of $\Om^p(X,\R)$ in a natural way -- each class in the quotient space has a unique harmonic representative.   Analogously, \emph{elements of the characteristic cohomology have unique representatives in normal form} \cite{Bryant1995}.
\begin{defn}
Let $\mcc \subset \Om^2(\M{\infty},\R) \cap \mci^{(\infty)}$ denote the space of representatives of differentiated conservation laws in normal form.  A conservation law on $\M{\infty}$ in normal form is said to have {\bf level $k$} if it is defined on $\M{k}$.   Let $\mcc_{(k)}$ denote the space of representatives of conservation laws of level $k$ in normal form.
\end{defn}
Clearly we have $\mcc_{(k)} \subset \mcc_{(k+1)}$ and $\mcc=\bigcup_k \mcc_{(k)}$.
In a series of lemmas we will prove the following proposition, in which the normal form is further refined. In this section, we will not prove the existence of elements of $\mcc_{(k)}$; the proposition only tells us what the elements of $\mcc_{(k)}$ must look like if they do exist.
\begin{prop}\label{prop:NormalForm}
Any element of $\mcc_{(k)}$ is of the form
\begin{equation}\label{eq:NormalForm}
\Phi  =\eta_0 \w \rho + A \psi +\sum_{1\leq i < j \leq k } \left( B^{ij}\eta_i \w \eta_j +\ol{B}^{ij}\etab_i \w \etab_j\right).
\end{equation}
The one--form $\rho$ and the function $B$ are determined by $A$ via the formulae
\begin{align}
\rho &= -\frac{1}{2} J\ed A  \label{eq:rhobyA}\\
B^{ij}&=\mo \sum_{m=0}^{k-j-i+1}(-1)^{m-i+1}{m+i-1 \choose i-1}(e_{-1})^{m}A_{u_{m+j+i-1}}\label{eq:BbyA}
\end{align}
if we insist on the normalization $e_0 \lhk \rho=0$, which we will do.  The function $A$ on $\M{k}$ -- which we henceforth refer to as the {\bf generating function} of $\Phi$ -- satisfies 
\be\label{eq:NoMixing}
A_{u_i,\ol{u_j}}=A_u=0
\ee 
and   
\begin{equation}\label{eq:ED}
\mce(A):=e_{\ol{-1}} e_{-1}A+f_u A=0.
\end{equation}
\end{prop}
\begin{rem}
The normal form of the differentiated conservation laws can be anticipated.  From its definition, $\psi \in \Om^{(2,0)}(\M{k})\oplus \Om^{(0,2)}(\M{k})$ so that modulo $\eta_0$, $\Phi  \in  \Om^{(2,0)}(\M{k})\oplus \Om^{(0,2)}(\M{k})$.  When $f=0$ Equation \eqref{eq:fGordon} is Laplace's equation, and then \eqref{eq:EDS} is an integrable extension of the EDS for holomorphic curves in $\C^2$.  Differentiated conservation laws of the EDS for holomorphic curves in $\C^n$ are closed forms in $\Om^{(2,0)}(\C^n)\oplus \Om^{(0,2)}(\C^n)$ \cite{Bryant1995}.   
\end{rem}

\begin{rem}
Whenever we use $B^{ij}$ with $i,j$ outside the index range $1\leq i<j\leq k$, we understand that $B^{ij}=0$.
\end{rem}
Let $\Phi \in \mcc_{(k)}$. To prove the proposition we unravel the consequences of $\ed \Phi=0$.  First we examine the weaker condition $\eta_0 \w \ed \Phi=0$.
\begin{lem}\label{lem:2002}
A conservation law $\Phi$ is of type $(2,0) + (0,2)$ modulo $\eta_0$; in other words, $D^{ij}=0$ for all $1 \leq i,j \leq k$, and thus $\Phi$ is of the form \eqref{eq:NormalForm}.
\end{lem}
\begin{proof}
For $i=1\ldots k$, the $\eta_0 \w \zetab \w \eta_i \w \etab_{k+1}$--coefficient of $\eta_0 \w \ed \Phi$ is $D^{ik}$, so $D^{ik}=0$.  Now assume that $D^{i,k-r}=0$ for $r=0 \ldots j$ and $i \leq k-r$.  We show that $D^{i,k-j-1}$ for $i \leq k-j-1$.  The coefficient of $\eta_0 \w \zetab \w \eta_i \w \etab_{k-j}$ when $i<k-j-1$ is $D^{i,k-j-1}$ plus terms that vanish by the induction hypothesis.  When $i=k-j-1$ the coefficient is $D^{k-j-1,k-j-1}-\ol{D}^{k-j-1,k-j-1}$.  But $D^{l,l}$ is imaginary since $\Phi$ is real.  
\end{proof}

\begin{lem}\label{lem:lemmaba}
\begin{align*}
\rho& \equiv -\frac{1}{2}J\ed A \equiv -\frac{\mo}{2}\left( \del A- \delb A \right) \; \mod \eta_0 \\
A_{u_{j-1}}&=-\mo (e_{-1}B^{1j}+ B^{1,j-1})  \hspace{1cm} j =2\ldots k\\
A_{u_k}&=-\mo B^{1k}
\end{align*}
\end{lem}

\begin{proof}
We express $\rho$ in the standard coframe as 
\be\nonumber
\rho=\rho^0 \eta_0+ \rho^{-1}\zeta+ \rho^{\ol{-1}}\zetab+\sum_{i=1}^{k}( \rho^{i}\eta_i+ \rho^{\ol{i}}\etab_i).
\ee
Taking $\rho^0=0$, which we are free to do, we calculate 
\begin{align}\label{eq:hookhook}
0 &= e_{1} \lhk e_{-1} \lhk \ed \Phi \equiv \rho - \rho^{-1}\zeta - \rho^1 \eta_1 -\frac{\mo}{2} (\ed A -e_{-1}(A)\zeta - A_{u_0} \eta_1 )\\
& + \sum_{j=2}^k e_{-1}B^{1j}\eta_j + \sum_{j=2}^{k-1} B^{1j}_{}\eta_{j+1}+ B^{1k}\eta_{k+1}  \;\;\; {\rm mod}\;\; \eta_0. \nonumber
\end{align}
This implies that 
\[ \rho^{(0,1)}=\frac{\mo}{2}\delb A.  \] 
The first statement follows from the identities $\rho \equiv \rho^{(1,0)}+\rho^{(0,1)}\mod \eta_0$ and $\ol{\rho^{(0,1)}}=\rho^{(1,0)}$. The other two relations follow from the vanishing of the coefficients of $\eta_i,\eta_{k+1}$ in Equation \eqref{eq:hookhook}.
\end{proof}
\begin{lem}\label{lem:afromb}
We have $A_u=0$ and $A_{u_i,\ol{u_j}}=0$ for all $i,j$, i.e.~\eqref{eq:NoMixing} is true.
\comment{ \begin{enumerate}
\item $A_0=0$ \label{eq:nou}, $B^{ij}_0=0$\;\; {\rm for all}\; $i,j>0$ 
\item $A_{i,\ol{j}}=0$ for $i,j \geq 1$ \label{eq:anomixing}
\item $B^{ij}_{k+1}=0$ for $i>1$
\item $B^{ij}_{\ol{l}}=0$ for $1 \leq l \leq k+1$
\end{enumerate}
Thus the functions $B^{ij}$ only depend on the $u_i$, not on the $\ub_i$. }
\end{lem}
\begin{proof} The coefficient of $\eta_0 \w \zeta \w \eta_1$ in $\ed \Phi$ is $\frac{\mo}{2}(-e_1e_{-1}A+e_{-1}e_1A -A_u)$. Together with the commutator $[e_{-1},e_{1}]A=-A_u$ this proves $A_u=0$. The coefficient of $\eta_0 \w \eta_i \w \etab_j$ is $\frac{\mo}{2}(-e_{\ol{j}}e_iA-e_ie_{\ol{j}}A)$. Combined with the commutator $[e_i,e_{\ol{j}}]=0$ this proves the second claim.
\end{proof}

The vanishing $A_u=0$ allows us to express $\rho$ exactly:
\begin{cor}
$\rho= -\frac{1}{2}J\ed A$, i.e.~\eqref{eq:rhobyA} is true.
\end{cor}

\begin{lem}\label{lem:BFormula}
The $B^{ij}$ are given by Equation \eqref{eq:BbyA}\upshape{:}
\begin{align*}
B^{ij}=\mo \sum_{m=0}^{k-j-i+1}(-1)^{m-i+1}{m+i-1 \choose i-1}(e_{-1})^{m}A_{u_{m+j+i-1}}.
\end{align*}
Therefore if $A$ is weighted homogeneous (cf.~\eqref{eq:wdegforms}), then $B^{ij}$ is weighted homogenous and $\wt(B^{ij})=\wt(A) -i-j$.
\end{lem}
\begin{proof} 
First of all note that for $1<i \leq j\leq k$, the $\eta_0 \w \zeta \w \eta_i \w \eta_{j+1}$--coefficient in $\eta_0 \w \ed \Phi=0$ is
\be \label{eq:Brecursive}
B^{ij}+B^{i-1,j+1}+e_{-1}B^{i,j+1}=0.
\ee
In particular, $B^{ik}=0$ for $i>1$ which is compatible with the formula to be proven.

We prove the lemma by induction on $i$. For $i=1$, we have to show that for any $j=2\ldots k$
\be\label{eq:B1j}
B^{1j}=\mo \sum_{m=0}^{k-j} (-1)^m (e_{-1})^m A_{u_{m+j}},
\ee
which we prove by induction on $j$, going down from $k$ to $2$. For $j=k$, the right--hand side is $\mo A_{u_k}$, which equals $B^{1k}$ by the third item of Lemma \ref{lem:lemmaba}. Assume we have shown \eqref{eq:B1j} for some $j$, then by the second item of Lemma \ref{lem:lemmaba},
\begin{align*}
B^{1,j-1}&=\mo A_{u_{j-1}} - e_{-1}B^{1j}\\
&=\mo A_{u_{j-1}} + \mo \sum_{m=0}^{k-j} (-1)^{m+1} (e_{-1})^{m+1} A_{u_{m+j}}\\
&=\mo \sum_{m=0}^{k-j+1} (-1)^m (e_{-1})^m A_{u_{m+j-1}}.
\end{align*}
Assume now that $i>1$ is such that the formula for $B^{i-1,j}$ is true for all $j=1\ldots k$. We will prove the formula for $B^{ij}$ by induction on $j$ as in the case $i=1$. Above, we argued that it is correct for $j=k$, and assuming that $j$ is such that the formula for $B^{i,j+1}$ is correct we use \eqref{eq:Brecursive} to prove the formula for $B^{ij}$:
\begin{align*}
\mo B^{ij}&=-\mo B^{i-1,j+1}- \mo e_{-1}B^{i,j+1}\\
&=(-1)^iA_{u_{j+i-1}}\\
&+\sum_{m=1}^{k-j-i+1} (-1)^{m-i} \left[ {m+i-2\choose i-2} +{m+i-2\choose i-1}\right] (e_{-1})^m A_{u_{m+j+i-1}}\\
&=\sum_{m=0}^{k-j-i+1} (-1)^{m-i} {m+i-1 \choose i-1} (e_{-1})^mA_{u_{m+j+i-1}}.
\end{align*}
\end{proof}

The following unassuming corollary has important consequences unraveled in Section \ref{sec:Existence}.
\begin{cor} \label{cor:B1kvanishes} If $k$ is odd then $B^{1k}=A_{k+1}=0$.
\end{cor}
\begin{proof} For $i+j=k+1$, Lemma \ref{lem:BFormula} gives
\[
B^{ij}=(-1)^{i+1} \mo  A_{u_k}.
\]
Writing $k+1=2n$ and choosing $i=j=n$, $B^{n,n}=0$ implies that $A_{u_k}=0$ and consequently, taking $i=1$ and $j=k$, that $B^{1k}=0$.
\end{proof}

Finally we deduce \eqref{eq:ED}.  The coefficient of $\eta_0 \w \zeta \w \zetab$ in $\ed \Phi=0$ is
\[
e_{-1}e_{\ol{-1}}A+e_{\ol{-1}}e_{-1}A+2f_u A=0.
\]
Since we are working on $\M{\infty}$, we have $[e_{-1},e_{\ol{-1}}]A=0$, which allows us to rewrite this as
\[
e_{\ol{-1}}e_{-1}A+f_u A=0.
\] 
This completes the proof of Proposition \ref{prop:NormalForm}.

\section{Homogeneity and a normal form for undifferentiated conservation laws}\label{sec:Homogeneous}
Exterior differentiation commutes with the decomposition 
\[
\Om^p(\M{\infty},\C)=\bigoplus_{d\in \Z} \Om^p_{d }(\M{\infty}),
\]
where the space $\Omega^p_d(\M{\infty})$ of differential $p$--forms of homogeneous weighted degree $d$ was defined in \eqref{eq:wdegforms}. Let $\mcc_{d} $ be the image of the projection $\pi_d$ from $\mcc \subset \Om^2(\M{\infty},\R)\subset \Om^2(\M{\infty},\C)$ to $ \Om^2_d(\M{\infty})$. 
\begin{lem} The $\mcc_d$ are complex subspaces of $\Omega^2_d(\M{\infty})$. Furthermore, 
$\mcc\otimes\C = \bigoplus_d \mcc_{d}$ and if $\Phi_d\in \mcc_d$, then $\Phi_d+\ol{\Phi_d}\in \mcc$.
\end{lem}
\begin{proof}
As mentioned above, any representative of a differentiated conservation law in normal form can be decomposed into weighted homogeneous pieces, so we have $\mcc\subset \bigoplus_d \mcc_{d}$. 

Let $\Phi_d \in \mcc_d$. By definition, $\Phi_d=\pi_d(\Phi)$ for some representative $\Phi\in \mcc$. Since $\Phi$ is a real--valued form and $\wt(\Psi)=-\wt(\ol{\Psi})$ for all weighted--homogeneous forms $\Psi$, we have $\pi_{-d}(\Phi)=\ol{\Phi_d}$. Since summands of different weighted degree cannot cancel, it follows that $\Phi_d+\ol{\Phi_d}$ is in normal form and hence an element of $\mcc$. But then, for any $b\in \C$ it follows that $b\Phi_d+\ol{b\Phi_d}\in \mcc$, so $b\Phi_d\in \mcc_d$, and $\mcc_d$ is a complex subspace. 

In the last argument, taking $b=\mo$ implies that $\mo \Phi_d-\mo \ol{\Phi_d} \in \mcc$, so that $ \Phi_d- \ol{\Phi_d} \in \mcc \otimes \C$.  Therefore $\Phi_d=\frac12(\Phi_d+\Phi_d)+\frac12(\Phi_d-\Phi_d)\in \mcc\otimes \C$ and $\bigoplus_d \mcc_d \subset \mcc \otimes \C$.  From  $\mcc \subset \bigoplus_d \mcc_{d}$ it follows that $\mcc\otimes\C \subset \bigoplus_d \mcc_{d}$ and so we can conclude that $\mcc\otimes\C = \bigoplus_d \mcc_{d}$.
\end{proof}

Given a conservation law $\Phi\in \mcc$ in normal form 
\[
\Phi  =\eta_0 \w \rho + A \psi + B^{ij}\eta_i \w \eta_j +\ol{B}^{ij}\etab_i \w \etab_j
\]
as in Proposition \ref{prop:NormalForm}, and writing $A=\sum_{d \geq 0}(P_d+\ol{P_d})$ with $\wt(P_d)=d$ then 
\[ \Phi_{P_d}:=\pi_d(\Phi)=\eta_0 \w \rho+{P_d}\psi + B^{ij}(P_d)\eta_i \w \eta_j  +\ol{B^{ij}(\ol{P_d})}\etab_i\w\etab_j \in \mcc_{d}\]
with $\rho=-\frac{1}{2}J \ed P_d$ and $B^{ij}(P_d)$ and $B^{ij}(\ol{P_d})$ being given by \eqref{eq:BbyA} using $P_d$, resp.~$\ol{P_d}$, in place of $A$ in the formula.  Using the weighted homogeneity of $\Phi_{P_d}$ we will produce a canonical representative of a class in $\hb_{\C}$ from the normal form of a class in $\mcc_{d}$.  To simplify notation we drop the subscript $d$ on $P$ but continue to assume that $\wt(P)=d$.

Let $F$ be the $\s{1}$--action defined in Equation \eqref{eq:Symmetry} and $v=\left.\frac{dF}{dt}\right|_{t=0}$ where $\lambda=e^{it}$.  One can calculate directly that
\be\nonumber
v=i\left(q e_0  +\zb e_{\ol{-1}} -z e_{-1} +(e_{-1})^j(q)e_j+ (e_{\ol{-1}})^j(q)e_{\ol{j}} \right),
\ee
where  $q=zu_0-\overline{zu_0}$.
For $\wt(P)=d\neq 0$, define 
\begin{equation}\label{eq:phiP}
\vp_P=\frac{1}{d} \left( v \lhk \Phi_P \right).
\end{equation}

\begin{lem}\label{lem:ifPhiclosed}
If $\wt(P)=d \neq 0$ and $\ed \Phi_P=0$ then $\Phi_P=\ed \vp_P$.
\end{lem}
\begin{proof}
Suppose $\Phi_p$ is closed and homogeneous. Then
\[
d \cdot \Phi_P=\dd{(F^*\Phi_P)}{t}\vert_{t=0}=\mcl_v \Phi_P=\ed (v \lhk \Phi_P).
\]
\end{proof}
The formulas for $\vp_P$ and $v$ lead to
\be\label{eq:phiexplicit}
\vp_P \equiv \frac{\mo}{2d}J\left( P \ed q- q \ed P  \right) \;\;\;{\rm mod} \;\; \I{\infty},
\ee
which we use in Lemma \ref{lem:Existence}.
\begin{rem}
It is simple to check that, given any function $G$ on $\M{\infty}$ satisfying \eqref{eq:ED} \us{(}but not necessarily \eqref{eq:NoMixing}\us{)}, then $[J(q\ed G-G \ed q)] \in \hb$.  It remains to show that it is a nontrivial element.  Furthermore, one still obtains a conservation law if one replaces $q$ with any solution to \eqref{eq:ED}.  This structure is closely related to the Poisson bracket defined in Theorem 4 of \cite{Shadwick1980} though we do not pursue this further here.
\end{rem}

We can now define canonical representatives for elements of $\hb$.  For $d \neq 0$ let $\mch_d^1$ be the image of the linear map
\begin{align*}
\mcc_d &\to \Om_d^1(\M{\infty},\C)\\
\Phi_P &\mapsto \vp_P
\end{align*}
and let $\mch^1_0 = \R \cdot \vp_0$, where $\vp_0$ is defined in \eqref{eq:wd0}.  Then let $\mch_{\C}^1=\bigoplus_d \mch^1_d$ and $\mch^1=\mch^1_{\C} \cap \Om^1(\M{\infty},\R)$.  
\begin{defn}
The {\bf  normal form} for an undifferentiated conservation law in $\hb$ is the representative $\vp \in \Om^1(\M{\infty})$ lying in $\mch^1$.  
\end{defn}
\begin{rem}
It would be interesting to find a definition of $\mch^1$ that is independent of $\mcc$. 
\end{rem}
\begin{rem}
The elements of $\mch^1$ are not invariant under translations in a lattice in the $z$--plane, even if $u(z,\zb)$ is.  One can prove that there are translation invariant representatives that therefore induce cohomology classes on the torus domains of doubly periodic solutions $u(z,\zb)$.  We will report on this and its implications in a forthcoming article. 
\end{rem}

\section{The space of conservation laws}\label{sec:Existence}
So far we have seen that $u_0$ and $q=z u_0 -\zb \ub_0$ are solutions to Equations \eqref{eq:NoMixing} and \eqref{eq:ED}.\comment{\footnote{We regard these as equations on $\M{\infty}$ though one could equally consider them on any $\M{k}$ with $k>l$ and $A$ defined on $\M{l}$.}}   These equations preserve weighted homogeneity and so to understand their solutions it is enough to understand the weighted homogeneous solutions.  
\begin{defn}
Let $V_d$ be the space of solutions to $\mce(P)=0$ (Equation \eqref{eq:ED}) of weighted degree $d$ that also satisfy $P_{u_i,\ol{u}_j}=P_u=0$.
\end{defn}
\begin{exam}
It is easy to check that $u_0 \in V_1$, $\ub_0 \in V_{-1}$, and $q \in V_0$.
\end{exam}

In this section we prove\comment{\marginpar{Is the space $\hb_d$ well define? only ask for homogeneity mod ideal?  or do hodge theory for it too?}}
\begin{thm}\label{thm:MainTheorem}
Suppose that $f$ does not satisfy a first--order ODE.  Then:
\begin{enumerate}
\item $V_0$ is spanned by $q$. If $d$ is a nonzero even integer, then $V_d=0$. If $d$ is odd, then $\dim V_d\leq 1$.
\item For all $d$ we have isomorphisms
\begin{align*}
&V_d \to \mch^1_d \to \mcc_d\\
&P \mapsto \vp_P \mapsto  \Phi_P,
\end{align*} 
where $\vp_P$ is defined as in Section \ref{sec:Homogeneous}, and the second map is just the exterior derivative.
\item $\dim_{\R}(\mcc_{(2n+1)}/\mcc_{(2n)}) =0$.
\item $\dim_{\R}(\mcc_{(2n+2)}/\mcc_{(2n)}) \leq 2$ with equality if and only if $\dim_{\C}(V_{2n+3})=1$.
\end{enumerate}
\end{thm}

We prove this theorem via the series of Lemmas \ref{lem:NoZMixing}--\ref{lem:Vdevenvanishes}.  Let $P\in V_d$ and write $P={U}(z,\zb,u_j)+{V}(z,\zb,\ub_j)+R(z,\zb)$, such that neither ${U}$ nor ${V}$ have any terms that don't involve at least one $u_j$ or $\ub_j$. 
We calculate that
\begin{align}
\mce({U})&=f_u{U} +{U}_{z \zb}+ u_{j+1}{U}_{u_j,\zb}-T^l \dd{}{u_l}\left({U}_z+u_{j+1}{U}_{u_j} \right) \label{eq:edU}\\
\mce({V})&=f_u{V} +{V}_{z \zb}+ \ub_{j+1}{V}_{\ub_j,z}-\ol{T}^l \dd{}{\ub_l}\left({V}_{\zb}+\ub_{j+1}{V}_{\ub_j} \right) \label{eq:edV}
\end{align}

\begin{lem}\label{lem:NoZMixing}
${U}_{\zb}={V}_z=R=0$. 
\end{lem}
\begin{proof}
The terms in $\mce(P)=0$ that don't involve $u$ imply that
\be
R_{z \zb}+{U}_{z \zb}+ u_{j+1}{U}_{u_j,\zb}+{V}_{z \zb}+ \ub_{j+1}{V}_{\ub_j,z}=0.
\ee
Let $u_k$ be the variable of highest weighted degree in $P$ that appears multiplied with a $\zb$ and $\ub_m$ the variable of lowest weighted degree appearing in $P$ with a $z$.  Then $u_{k+1}{U}_{u_k \zb}$ produces a monomial which, because of the maximality of $u_k$, can't be canceled by any of the other terms.  Then by induction ${U}_{u_j \zb}=0$ for all $u_j$.  A similar argument shows that ${V}_{\ub_j z}=0$ and because ${U}$ and ${V}$ don't have any terms without some $u_j$ or $\ub_j$, we have that ${U}_{\zb} ={V}_z =0$.  This then implies that $R_{z \zb}=0$.

Thus, the only remaining possibility for $R$, if it does not vanish, is that is consists of exactly one monomial of the appropriate degree. In the case $d>0$ it follows that $R=c\zb^d$ for some constant $c$. We have $\mce(R)=cf_u \zb^d$, so the negative of this term has to appear in $\mce(U+V)$. But from \eqref{eq:edU} and \eqref{eq:edV} we see that the only possibilities to get a summand in $\mce(U+V)$ without any $u_j$ are the terms $-T^0 \dd{}{u_0} U_z=-fU_{z,u_0}$ and $-fV_{\zb,\ub_0}$. Since $f$ and $f_u$ are not linearly dependent, it follows $R=0$. The same argument works for $d<0$ and $d=0$.
\end{proof}
We have $P=U(z,u_j)+V(\zb,\ub_j)$, where $U$ and $V$, expressed as power series in $z$ and $\zb$, can be written as $U=\sum U^n z^n$ with $\wt(U^n)=d+n$ and $V=\sum V^n \zb^n$ with $\wt(V^n)=d-n$. Each coefficient $U^n$ or $V^n$ is a polynomial in the $u_j$ or the $\ub_j$, never constant.  Now we can expand $\mce(P)=0$ in terms of $z$ and $\zb$:
\begin{align}
\mce(U)&=z^n\left[  f_u U^n-T^l \dd{}{u_l}\left((n+1)U^{n+1}+u_{j+1}U^n_{u_j} \right)\right] \label{eq:Un} \\
\mce(V)&=\zb^n \left[ f_u V^n -\ol{T}^l \dd{}{\ub_l}\left((n+1)V^{n+1}+\ub_{j+1}V^n_{\ub_j} \right)  \right] \label{eq:Vn}
\end{align}
The $n=0$ coefficients only sum to be zero, but otherwise the coefficients of $z^n$ and $\zb^n$ must vanish separately. 

\begin{lem}\label{lem:LinearInU} If $u_k$ is the variable of highest weighted degree appearing in $U$, then $U_{u_k,u_j}=0$ for all $u_j$.  Similarly, if $\ub_k$ is the variable of  lowest weighted degree appearing in $V$, then $V_{\ub_k,\ub_j}=0$ for all $\ub_j$.
\end{lem}

\begin{proof}
Let $u_k$ be the variable of highest weighted degree appearing in $U$.  Suppose that there is a summand in $U^n$ where $u_k$ appears to a power higher than $1$ or multiplied by some other $u_j$. Denote the monomial in $U^n$ of highest lexicographic ordering with this property by $u_{j_1}^{i_1}\cdot\ldots\cdot u_{j_r}^{i_r}$, with $k=j_1> \ldots> j_r$ and all exponents $\geq 1$. Our assumption says that either $r\geq 2$ or $r=1$ and $i_1\geq 2$. 

Let us look at the case $r\geq 2$ first: By finding a non-vanishing summand in \eqref{eq:Un}, we will derive a contradiction. Exactly for $j=j_1(=k)$ there appear summands involving $u_{k+1}$ in \eqref{eq:Un}: exactly those of the form $-T^l \dd{}{u_l}\left(u_{k+1}U^n_{u_k} \right)$ with $l\neq k+1$. Our monomial above produces
\[
-n_1 T^l \dd{}{u_l} \left[u_{j_1+1}u_{j_1}^{i_1-1}\cdot u_{j_2}^{i_2}\cdot\ldots\cdot u_{j_r}^{i_r}\right].
\]
For $l={j_r}$ we obtain
\begin{align*}
&-i_1i_rT^{j_r}u_{j_1+1}u_{j_1}^{i_1-1}\cdot u_{j_2}^{i_2}\cdot\ldots\cdot u_{j_r}^{i_r-1}\\
&\quad =\begin{cases} -i_1i_rf_u\left[u_{j_1+1}u_{j_1}^{i_1-1}\cdot u_{j_2}^{i_2}\cdot\ldots\cdot u_{j_r}^{i_r-1}u_{j_r-1}\right]+\text{lower lex.~ord.} & j_r> 0\\
-i_1i_r f\left[ u_{j_1+1}u_{j_1}^{i_1-1}\cdot u_{j_2}^{i_2}\cdot\ldots\cdot u_{j_r}^{i_r-1} \right]+\text{lower lex.~ord.} & j_r=0
\end{cases}
\end{align*}
and because the original monomial was the one of highest lexicographic ordering among those in $U^n$, this monomial cannot be canceled by any other that is produced from $U^n$ in \eqref{eq:Un}. But it also cannot cancel with a summand coming from $U^{n+1}$ since that would contradict our assumption that $u_k$ is the variable of highest weighted degree appearing in all of $U$.   Thus $u_k$ cannot appear in a monomial with any other $u_j$.

Now suppose that $r=1$ and $m=i_1\geq 2$ so that the highest monomial is $u_k^m$. Taking $j=l=k$ gives a summand of \eqref{eq:Un} of the form
\begin{align*}
-mT^k \dd{}{u_k}& \left[ u_{k+1}u_k^{m-1}\right] = -m(m-1) T^k u_{k+1}u_k^{m-2} \\ 
&=\begin{cases} -m(m-1)f_u\left[u_{k+1}u_k^{m-2}u_{k-1}\right]+\text{lower lex.~ord.} & k>0 \\
-m(m-1)f\left[u_{k+1}u_k^{m-2}\right] +\text{lower lex.~ord.} & k=0
\end{cases}
\end{align*}
which again is the unique highest one.  The same considerations as before lead to a contradiction.  Thus $u_k$ must appear linearly.  When it does the terms only involving $U^n$ allow it to cancel.

An analogous argument gives the corresponding result for $V$.
\end{proof}
\begin{cor}\label{cor:Highest}
If $u_k$ is the highest variable appearing in $U$ then it only appears in $U^{k+1-d}$, so $d \leq k+1$.  If $\ub_{m}$ is the lowest variable appearing in $V$ then it only appears in $V^{m+1+d}$, so that $d \geq -(m+1)$. 
\end{cor}
\begin{proof}
This follows by considerations of weighted degree and Lemma \ref{lem:LinearInU}.
\end{proof}
\begin{cor}
If $u_k$ is the highest variable appearing in $U$ then $U^n=0$ for $n >{k+1-d}$.  If $\ub_{m}$ is the lowest variable appearing in $V$ then $V^n=0$ for $n>{m+1+d}$.
\end{cor}
\begin{proof}
Let $u_l$ with $l < k$ be the highest variable appearing in those $U^n$ with $n>k+1-d$. If we regard such an $n$, the same argument as given in the proof of Lemma \ref{lem:LinearInU} implies that $u_l$ appears linearly and without any other $u_j$. This implies that $\wt(U^n)=l+1<k+1$, but this contradicts $\wt(U^n)=n+d>k+1$. A similar argument can be made for $V$.
\end{proof}
\begin{cor}
Both $U$ and $V$ are polynomials. In fact, 
\[
U=\sum_{n=0}^{k+1-d} U^n z^n \text{ with } U^{k+1-d}=b (u_k+\ldots) \text{ for some}\; b \in \C
\]
and
\[
V=\sum_{n=0}^{m+1+d} V^n \zb^n \text{ with } V^{k+1+d}=c( \ub_m+\ldots) \text{ for some } \;c\in \C.
\]
\end{cor}

We will make use of the following lemma repeatedly.
\begin{lem}\label{lem:eMinusOneKernel}
The operator $\ol{e_{-1}}$, acting on polynomials in $u_i$, has only the constants as kernel. 
\end{lem}
\begin{proof}
It suffices to prove that for a weighted-homogeneous polynomial $h$ of degree at least one, $\ol{e_{-1}}h=0$ implies $h=0$. Write
\[
h=\sum_{|I|=k+1} h_I u^I,
\]
where the sum runs over all multi-indices $I=(i_0,\ldots, i_k)$ of weighted degree $\sum_j (i_j+1)$ equal to $k+1$, and $u^I=u_0^{i_0}\cdot\ldots\cdot u_k^{i_k}$, and assume 
\begin{equation}\label{eq:proofe-1injective}
0=-\ol{e_{-1}}(h)=\sum_{j=0}^k \sum_{|I|=k+1} h_I T^j \frac{\partial}{\partial u_j} u^I.
\end{equation}
Let $I_0=(i_0,\ldots,i_k)\neq 0$ be the highest index such that $h_{I_0}\neq 0$.  Let furthermore $l$ be the smallest number such that $i_l\neq 0$, and assume first that $l>0$. In other words, $u^I=u_l^{i_l}u_{l+1}^{i_{l+1}}\cdot\ldots\cdot u_k^{i_k}$. In Equation \eqref{eq:proofe-1injective} we find $h_{I_0}$, for example, in the summand
\[
h_{I_0} T^l i_l \cdot u_l^{i_l-1}u_{l+1}^{i_{l+1}}\cdot\ldots\cdot u_k^{i_k}.
\]
(This is the summand for $I=I_0$ and $j=l$.) Since $T^l$ reads $T^l=u_{l-1}f_u+\text{terms with lower }u\text{'s}$, we have found a summand 
\[
h_{I_0}i_lf_u\cdot  u_{l-1}u_l^{i_l-1}u_{l+1}^{i_{l+1}}\cdot\ldots\cdot u_k^{i_k}.
\]
Let us try to spot the full coefficient of this monomial $u_{l-1}u_l^{i_l-1}u_{l+1}^{i_{l+1}}\cdot\ldots\cdot u_k^{i_k}$ in \eqref{eq:proofe-1injective}. For which $I$ and $j$ can the summand $h_I T^j \frac{\partial}{\partial u_j} u^I$ contribute? If $j>0$, some of the $u$'s in the monomial have to appear in $T^j$. But then, necessarily $I\ge I_0$, since $u^I$ is differentiated with respect to $u_j$, and $u_j$ is higher than all the $u$'s appearing in $T^j$. For $I=I_0$, we already have found the one contribution, so since we assumed that $I_0$ is the highest multi-index such that $h_{I_0}\neq 0$, the only further summands that can contribute are those with $j=0$. Here we only have a new contribution if $u^I=u_0u_{l-1}u_l^{i_l-1}u_{l+1}^{i_{l+1}}\cdot\ldots\cdot u_k^{i_k}$. Denoting the corresponding multi-index by $I_1$, we have shown:
\[
0=h_{I_0}i_jf_u+h_{I_1}cf,
\]
where $c=1$ or $c=2$, depending on whether $l>1$ or $l=1$. Since $f$ and $f_u$ are linearly independent, $h_{I_0}=0$ (and also $h_{I_1}=0$), a contradiction.

In the case $l=0$, the monomial in question is $u_l^{i_l-1}u_{l+1}^{i_{l+1}}\cdot\ldots\cdot u_k^{i_k}$, and there is only the summand for $j=0$ and $I=I_0$ contributing to this monomial; we also conclude $h_{I_0}=0$.
\end{proof}

\begin{lem}
$\wt(U) \geq 0$ and $\wt(V) \leq 0$. 
\end{lem}
\begin{proof}
Suppose that $\wt(U)=d < 0$.  We know that $U=\sum_{n=0}^{k+1-d}U^n z^n$. Since $\wt(U^n) \geq 0$, we have $U^0=0$; let $m \geq 0$ be such that $U^0=\ldots=U^m=0$ and $U^{m+1} \neq 0$.  If $m=0$, i.e.~$U^1\neq 0$, it would follow that $d=-1$ and that $U^1$ is constant, which was ruled out in Lemma \ref{lem:NoZMixing}.

So we are in the case $m>0$.  From $\mce(U+V)=0$ and \eqref{eq:Un} we find that 
\[\sum_{l}-z^m T^l \dd{}{u_l}\left((m+1)U^{m+1} \right)=0 \;\; {\rm (no \;sum\; on} \;m {\rm )}\]
 when $m >0$.  By Lemma \ref{lem:eMinusOneKernel} this implies that $U^{m+1}=0$ and so by induction $U=0$.
A similar argument gives the result for $V$.
\end{proof}
\begin{cor}
If $\wt(P)>0$, then $V=0$. If $\wt(P)<0$, then $U=0$.
\end{cor}
\begin{lem}\label{lem:V0}
If $\wt(P)=0$ then $P=b\cdot (z u_0 -\zb \ub_0)$ for some constant $b \in \C$.
\end{lem}
\begin{proof}
We have  
\[
U=b z u_0+\sum_{n \geq 2}^{k+1} U^n z^n,\qquad V=c \zb \ub_0+\sum_{n \geq 2}^{m+1} V^n \zb^n
\] 
for some constants $b,c$.
Then the first terms in \eqref{eq:Un} and \eqref{eq:Vn} lead to
\begin{align*}
\mce(U+V)&=z^0 \left(-f U^1_{u_0}-f V^1_{\ub_0} \right)\\
&+z^1 \left[ f_u U^1 - T^l \dd{}{u_l} \left(u_{j+1}U^1_{u_j}+2 U^2 \right)  \right]  \\
&+\zb^1 \left[ f_u V^1 - \ol{T}^l \dd{}{\ub_l} \left(\ub_{j+1}V^1_{\ub_j}+2 V^2 \right)  \right]  +\ldots
\end{align*}
The $z^0$--term implies that $c=-b$.  The terms from $U^1$ in the $z^1$--term cancel so that by Lemma \ref{lem:eMinusOneKernel} we have $U^2=0$.  Similarly, the $\zb^1$--term implies that $V^2=0$.  Then, using induction, \eqref{eq:Un} and \eqref{eq:Vn} imply that $U^n=V^n=0$ for $n>1$.
\end{proof}

\begin{lem}\label{lem:shapeofP}
If $\wt(P)=d>0$, then $P_z=0$, i.e.~$P$ is a polynomial in the $u_j$ of the form $P=b u_{d-1}+\dots$ with $b\neq 0$.  If $\wt(P)=d<0$, then $P_{\zb}=0$, i.e.~$P$ is a polynomial in the $\ub_j$ of the form $P=c \ub_{-d-1}+\dots$ with $c\neq 0$.  
\end{lem}
\begin{proof}
Suppose $d>0$ so that $P=\sum_{n=0}^{k+1-d}U^n z^n$. Assuming that $k+1-d>0$ will lead to a contradiction. We have 
\[
U^{k+1-d}=b u_{k}+b' u_{k-1}u_0+\ldots
\] with $b\neq 0$, and  
\[
U^{k-d}=b'' u_{k-1}+\ldots
\]  
The $z^{k-d}$--coefficient in \eqref{eq:Un} reads
\begin{align*}
f_u U^{k-d}-T^l\dd{}{u_l}\left( (k-d+1) U^{k-d+1} + u_{j+1} U^{k-d}_{u_j}\right);
\end{align*}
neglecting summands without $u_{k-1}$ and, finding that the $b''$ terms cancel, we obtain
\begin{align*}
0= - (k-d+1) (T^0 b'u_{k-1} + T^k b) =-(k-d+1) u_{k-1}(fb'+f_ub)+\ldots
\end{align*}
Since $f$ and $f_u$ are linearly independent, $(k-d+1)>0$ and $b\neq 0$, we have arrived at a contradiction. A similar argument works for $V$.
\end{proof}

\begin{cor}\label{cor:dimVdbound}
For all $d$, we have $\dim_{\C}(V_d) \leq 1$. 
\end{cor}
\begin{proof}
We know that any nonzero element in $V_d$ is of the form $b u_{d-1}+\ldots$ with $b\neq 0$ for $d>0$, or $c \ub_{d-1}+\ldots$ with $c\neq 0$ for $d<0$ and $a \cdot (zu_0-\zb \ub_0)$ with $a\neq 0$ for $d=0$.  The bound on dimension then follows because $\mce(P)=0$ is a linear equation.  \end{proof}

\begin{lem}\label{lem:Existence}
For all $d$ the linear map
\begin{align*}
&V_d \to \hb_{\C}  \\
& P \mapsto [\vp_P]
\end{align*}
is injective.
\end{lem}
\begin{proof} 
Assume that $d>0$.  The $d<0$ case follows by complex conjugation.  Let $P\in V_d$ be nonzero and normalized, i.e.~$P$ is a polynomial of the form 
\[
P=u_{d-1}+\ldots.
\]
The $1$--form $\vp_P$ was defined in \eqref{eq:phiP}, and a more explicit form was given in \eqref{eq:phiexplicit}.
It will be convenient to modify $\vp_P$ by an exact form as follows:  If we define $\tilde{E}'=qe_{-1}(P)$,  $\tilde{E}''=e_{\ol{-1}}(q) P$ and 
\[
\tilde{\vp}_P= \tilde{E}' \zeta+ \tilde{E}'' \zetab,
\]
then
\[
\tilde{\vp}_P \equiv d\cdot \vp_P + \ed (qP) \; \mod \; \I{d}.
\]
First, we have to show that $\tilde{\vp}_P$ defines a cohomology class in $\hb_{\C}$.  The only obstacle that could arise is that $\ed \tilde{\vp}_P$ could have a $\zeta \w \zetab$ term.  However the corresponding coefficient is 
\begin{align*}
-e_{\ol{-1}}\tilde{E}'+e_{-1}\tilde{E}'' &= -e_{\ol{-1}}(q)e_{-1}P-qe_{\ol{-1}}e_{-1}P + e_{\ol{-1}}e_{-1}(q)P+e_{\ol{-1}}(q)e_{-1}P\\
&=qf_uP-f_uqP=0
\end{align*}
 because $P\in V_d$ and $q\in V_0$.

Thus, it remains to show that $[\tilde{\vp}_P] =d\cdot[\vp_P]$ is a nontrivial class. The one--form $\tilde{\vp}_P$ represents $0 \in \hb_{\C}$ if and only if 
\be \label{eq:phialpha}
\ed \tilde{\vp}_P=\ed  \alpha
\ee
for some $\alpha \in \I{l}$.   Assuming that $\alpha=\sum_{j=0}^{l}(a^j \eta_j+b^j \etab_j)$ (with $b^0=0$) satisfies \eqref{eq:phialpha} will lead to a contradiction.   

For $j>1$, the $\zeta\w \eta_j$--coefficient of \eqref{eq:phialpha} implies 
\[
-qe_j e_{-1}P=-\tilde{E}'_{u_{j-1}} = e_{-1}a^j+a^{j-1}, 
\]
from which we can determine the $a^j$ recursively: We have $e_{-1}P=u_d+\ldots$, so $a^j$ vanishes for $j>d$. The first two non-vanishing coefficients are 
\be\nonumber
a^d=-qe_{d+1}e_{-1}P=-q
\ee
and
\be\label{eq:ad-1}
a^{d-1}=-e_{-1}a^{d}-q e_de_{-1}P=e_{-1}q-q e_de_{-1}P.
\ee
We obtain 
\begin{align}
a^1&=\sum_{j=2}^{d+1}(-1)^{j-1} (e_{-1})^{j-2}(qe_je_{-1}P) \label{eq:a1}\\
&= (-1)^d\left[ (e_{-1})^{d-1}(q)-(e_{-1})^{d-2}(qe_de_{-1}P)\pm\ldots +(-1)^{d-1}qe_2e_{-1}P\right]\nonumber
\end{align}
The condition on the $\eta_1\w \eta_{d-1}$--coefficients of \eqref{eq:phialpha} is
\be\nonumber
a^{d-1}_{u_0}=a^1_{u_{d-2}},
\ee
which will provide a contradiction.  One finds that  
\[
(e_{-1})^jq=ju_{j-1}+zu_j+\overline{z}T^{j-1}.
\]
Using Equation \eqref{eq:ad-1} we compute 
\[
a^{d-1}_{u_0}=1-ze_de_{-1}P-qe_1e_de_{-1}P,
\]
which has a constant term $1$. On the other hand, the only way to obtain a constant term from differentiating \eqref{eq:a1} with respect to $u_{d-2}$ is via the summand $(-1)^d (e_{-1})^{d-1}(q)$:
\[
a^1_{d-1}=(-1)^d d + \text{non-constant terms}.
\]
But $(-1)^d d\neq 1$ for all $d >0$.
\end{proof}

\begin{lem}\label{lem:dphi=Phi}
If $P\in V_d$ then $\ed \vp_P=\Phi_P$.
\end{lem}
\begin{proof}
The $d=0$ case was done explicitly in Section \ref{sec:Classical}, so assume $d \neq 0$.  If $P\in V_d$ is nonzero, then $[\vp_P]$ is a nontrivial class and so $[\ed \vp_P]=[\Phi_{P'}]$ for some other solution $P'$ to \eqref{eq:NoMixing} and \eqref{eq:ED}.  Weighted degree is preserved by exterior differentiation, so $\wt(P')=\wt(P)$ and hence $P'\in V_d$, which implies that $P'=c \cdot P$ for some constant $c$, by Corollary \ref{cor:dimVdbound}.  Since $\Phi_{P'}$ is closed, this implies that $\Phi_P$ is closed. Then by Lemma \ref{lem:ifPhiclosed} we reach the desired conclusion.
\end{proof}

\begin{cor}\label{cor:VdPhi}
If $P \in V_{k+1}$ then $\Phi_P$ is closed and $\Phi_P+\ol{\Phi_P}$ is a real element of $\mcc_{(k)}$.
\end{cor}

\begin{lem}\label{lem:Vdevenvanishes}
For even degree $d\neq 0$, $V_d=0$.
\end{lem}
\begin{proof}
For $d\neq 0$, complex conjugation gives an isomorphism $V_d\to V_{-d}$. Thus, it suffices to prove the lemma for positive $d$.

Suppose that $d=2n>0$ and $P=u_{d-1}+\ldots\in V_d$  is a normalized solution.  Then by Lemmas \ref{lem:Existence} and \ref{lem:dphi=Phi} the two--form 
\[
\Phi_P=\eta_0\w \rho_P + P\psi + \sum_{i,j} B^{ij}(P)\eta_i\w \eta_j\in \Omega^2_d(\M{d-1})
\] 
defines a weighted--homogeneous differentiated conservation law.   By Corollary \ref{cor:B1kvanishes} we can conclude that $B^{1,d-1}(P)=0$, which contradicts the third item of Lemma \ref{lem:lemmaba} due to the fact that, if such a $P$ exists, then $P=u_{d-1}+\ldots$.
\end{proof}

\begin{proof}[Proof of Thm \ref{thm:MainTheorem}]
The statements about $V_d$ are exactly Lemma \ref{lem:V0}, Corollary \ref{cor:dimVdbound}, and Lemma \ref{lem:Vdevenvanishes}.  By Corollary \ref{cor:VdPhi} the map $P \to \Phi_P$ is an isomorphism from $V_d$ to $\mcc_{d}$.  By definition the map $\mcc_d \to \mch^1_d$ is an isomorphism and Lemma \ref{lem:dphi=Phi} implies that its inverse is given by exterior differentiation.  The last two items are immediate consequences of the second item of the theorem.
\end{proof}

\section{Potentials satisfying linear second order ODEs}\label{sec:fuu=betaf}
So far, the only assumption on $f$ was that it does not satisfy a linear first--order ODE, i.e.~that $f$ and $f_u$ are linearly independent.\footnote{If $f_u=\beta f$ for some constant $\beta$, then \eqref{eq:fGordon} is the Liouville equation.  It is not hard to check that it has infinitely many classical conservation laws.  It is well known that the Liouville equation is linearizable.  With respect to its role as the Gauss equation for constant mean curvature surfaces with $\epsilon+\delta^2=0$ (see Equation \eqref{eq:Potentials}), the linearizability is equivalent to the existence of the Weierstrass representation.} The following theorem shows that $f$ has to satisfy a linear second--order ODE for higher-order conservation laws to exist.\footnote{It has been suggested to us that the classification of potentials that make equation \eqref{eq:fGordon} an `integrable system' has appeared repeatedly in the literature.  The only articles we have found so far that make such claims are that of Dodd and Bullough \cite{Dodd1977} and {\v Z}iber and {\v S}abat \cite{Ziber1979}. 

In \cite{Ziber1979} they list the equations of the form $u_{xt}=f(u)$ that admit a nontrivial Lie--B\"acklund transformation group. As the article only appears in Russian, we have not been able to study their method.  The article \cite{Ibragimov1980} by Ibragimov and {\v S}abat appears to use a similar method though, and from it we can see that the existence of a nontrivial Lie-B\"acklund group is equivalent to the existence of solutions to the linearized equation (see Section \ref{sec:Conclusion} of the current article for more on this).  We do not know if in \cite{Ziber1979} they prove that there are infinitely many non-trivial solutions or if they only show that for the equations in question some non-trivial solutions do exist.  We should also warn the reader that in his review of \cite{Ibragimov1980} on MathSciNet, Vinogradov claims that the proofs in the article have gaps, as does the theory on which it is based.  

Our work is much closer to the approach in \cite{Dodd1977} where they look for the existence of polynomial conserved quantities, though we don't rely on either \cite{Dodd1977} or \cite{Ziber1979}.  
}

\begin{thm}\label{thm:SecondOrder}Assume that $f$ does not satisfy any linear second order ODE, i.e.~that $f,f_u$ and $f_{uu}$ are linearly independent over $\R$. Then $V_d=0$ for $|d|\geq 2$, i.e.~no higher-order conservation laws occur.
\end{thm}
\begin{proof}
It suffices to prove $V_d=0$ for $d\geq 2$. Assume that $V_d\neq 0$ for some $d\geq 2$, and let $P\in V_d$ be a nonzero element. By Lemma \ref{lem:shapeofP}, $P$ is a polynomial in the variables $u_j$ and can be normalized such that 
\[
P=u_{d-1}+cu_{d-2}u_0+\ldots
\] The polynomial $P$ satisfies \eqref{eq:ED}
\be \label{eq:edinnewlemma}
\sum_{j=0}^d\sum_{i=0}^{d-1} T^j \frac{\partial}{\partial u_j} \left(u_{i+1}P_{u_i}\right) = f_u P.
\ee
Recall from Lemma \ref{lem:tidentities} that $T^j=(e_{-1})^jf$. It follows that $T^0=f$, $T^1=u_0f_u$, $T^2=u_1f_u+u_0^2f_{uu}$ and for $j\geq 3$, 
\[
T^j=u_{j-1}f_u+ju_{j-2}u_0 f_{uu}+\text{terms without }u_{j-1}\text{ and }u_{j-2}.
\]
Therefore, the summands on the left hand side of  \eqref{eq:edinnewlemma} that involve $u_{d-2}$ are
\begin{center}
\begin{tabular}{lll}
$j=0,i=d-3$, if $d\geq 3$ & & $fu_{d-2}P_{u_{d-3},u_0}$\\
$j=1,i=d-3$, if $d\geq 4$ & & $f_uu_{d-2}u_0 P_{u_{d-3},u_1}$\\
$j=1,i=0$ & & $f_u c u_{d-2}u_0$\\
$j=d-1, i=d-2$, if $d\geq 3$ & & $f_u c u_{d-2}u_0$ \\
$j=d, i=d-1$ & & $f_{uu}d u_{d-2}u_0$.
\end{tabular}
\end{center}
It follows that the vanishing of the $u_{d-2}u_0$--coefficient of \eqref{eq:edinnewlemma} contradicts the assumption that $f_{uu}$ is linearly independent from $f$ and $f_u$. Therefore, $V_d=0$. The statement about conservation laws follows from Proposition \ref{prop:NormalForm}.
\end{proof}

On the other hand, in certain cases the upper bound for the dimensions of the spaces of higher-order conservation laws given in Theorem \ref{thm:MainTheorem} is sharp. A slight modification of Proposition 3.1 of \cite{Pinkall1989} provides nontrivial elements of $V_{2n+1}$ when $f_{uu}=\beta f$.  
\begin{lem}\label{lem:PSSolutions}
Suppose that $f_{uu}=\beta f$ with $\beta\neq 0$.  Make the following recursive definitions:
\begin{align*}
P^1&=u_0\\
\phi^i&=\begin{cases} (P^l)^2+2\sum_{j=1}^{l-1} \theta^{j,i-j} & {\rm if} \;\; i=2l-1\\
P^{l+1}P^l+\theta^{l,l}+2\sum_{j=1}^{l-1} \theta^{j,i-j} & {\rm if} \;\; i=2l\end{cases}\\
\theta^{l,m}&=P^l P^{m+1}-e_{-1}(P^{l})e_{-1}(P^m)+\frac{\beta}{4}\phi^l \phi^m \\
P^{i+1}&=e_{-1}e_{-1}P^i-\frac{\beta}{2}u_0 \phi^i.
\end{align*}
Then
\be\label{eq:phiequation} 
\begin{split}
e_{-1}\phi^i&=2u_0 e_{-1}P^i\\
e_{\ol{-1}}\phi^i&=-2f P^i
\end{split}
\ee
and 
\be\label{eq:ed2}
e_{\ol{-1}}e_{-1}P^i=-f_u P^i,
\ee
so $P^i\in V_{2i+1}$.
\end{lem}
\begin{proof}
We calculate that
\begin{align*}
e_{\ol{-1}}P^{l+1}&=\frac{\beta}{2}f \phi^l-f_ue_{-1}P^l\\
e_{\ol{-1}}\theta^{l,m}&=e_{\ol{-1}}(P^l)P^{m+1}-P^me_{\ol{-1}}P^{l+1}\\
e_{-1}\theta^{l,m}&=P^l e_{-1}P^{m+1}-e_{-1}(P^m)P^{l+1}.
\end{align*}
Using these and induction we can verify the conditions in \eqref{eq:phiequation}.  Using this we can verify that $P^i$ satisfies \eqref{eq:ed2}.   
\end{proof}

\begin{cor}\label{cor:dimVdspecial}
Suppose that $f_{uu}=\beta f$ with $\beta \neq 0$ and that $f$ does not satisfy a first--order ODE.  Then $\dim V_d=1$ for all odd integers $d$.
\end{cor}

\begin{exam}
Note that the generators $P^i$ of $V_{2i+1}$ are normalized so that $P^i=u_{2i}+\ldots$. The first four of them are
\begin{align*}
P^1&=u_0\\
P^3&=u_2-\frac{1}{2} \beta u_0^3\\
P^5&=u_4-\frac52 \beta u_2 u_0^2 - \frac52 \beta u_1^2u_0 + \frac38 \beta^2 u_0^5\\
P^7&=u_6-\frac72 \beta u_4 u_0^2 - 14 \beta u_3u_1u_0 - \frac{21}{2} \beta u_2^2u_0 - \frac{35}{2} \beta u_2^2u_1^2+\frac{35}{8}\beta^2 u_2 u_0^4 + \frac{35}{4} \beta^2 u_1^2 u_0^3 \\
&\qquad\,\,- \frac{5}{16}\beta^3 u_0^7.
\end{align*}
\end{exam}

\begin{exam} In the case that $f_{uu}=\alpha f_u + 2\alpha^2 f$ with $\alpha\neq 0$ a coordinate change transforms \eqref{eq:fGordon} into the Tzitzeica equation $u_{z\bar{z}}=e^u-e^{-2u}$. For $f(u)=e^u-e^{-2u}$, the first $V_{2i+1}$ are as follows:
\begin{align*}
&\dim V_1=1,\qquad u_0\in V_1\\
&\dim V_3=0\\
&\dim V_5=1,\qquad u_4+5u_2u_1-5u_2u_0^2-5u_1^2u_0+u_0^5\in V_5\\
&\dim V_7=1,\qquad u_6+7u_4u_1-7u_4u_0^2+14u_3u_2-28u_3u_1u_0-21u_2^2u_0\\  &\qquad\qquad\qquad\quad\;\;-28u_2u_1^2-14u_2u_1u_0^2+14u_2u_0^4-\frac{28}{3}u_1^3u_0+28u_1^2u_0^3-\frac43 u_0^7\in V_7.
\end{align*}
Of course it would be nice to have a recursion for generators of $V_d$ in this case as well.  Such a recursion almost certainly exists and should be derivable using the work of Guthrie \cite{Guthrie1994}.  However, Dodd and Bullough claim to that there are only finitely many polynomial conserved quantities in this case in \cite{Dodd1977}.    This would be strange since both families of potentials discussed in this section are obtained from primitive maps into $k$-symmetric spaces and thus should have similar theories.  An existence result and method for calculating generators may also be derivable from formal Killing fields by using the approach of Terng and Uhlenbeck in \cite{Terng2000}.   It isn't clear if the remarkable recursive formula presented in \cite{Pinkall1989} and copied here in Lemma \ref{lem:PSSolutions} has an analogue for the Tzitzeica equation or for the other systems of PDEs associated to primitive maps.   
\end{exam}

\begin{rem}
We believe that $f$ must satisfy either $f_{uu}=\beta f$ or $f_{uu}=\alpha f_u + 2\alpha^2 f$ if higher-order conservation laws exist.  It is not hard to show that for there to exist new conservation laws in normal form at the second prolongation,  then $f$ must satisfy $f_{uu}=\beta f$, and for there to exist new conservation laws in normal form at the fourth prolongation then $f$ must satisfy either $f_{uu}=\beta f$ or $f_{uu}=\alpha f_u + 2 \alpha^2 f$.  We do not have a proof that $f(u)$ must satisfy one of these two second-order ODE's in order for higher-order conservation laws to appear at any level.
\end{rem}

\section{Generalized symmetries}\label{sec:Symmetries}
Noether's theorem can be formulated as an isomorphism between the space of proper conservation laws (viewed as elements of the characteristic cohomology) and the space of proper generalized symmetries \cite{Bryant2003}.  See, for example, \cite{Olver1993,Shadwick1980} for related formulations of Noether's theorem. In order to discuss this for the system at hand, we begin by introducing the appropriate class of generalized symmetries. In Lemma \ref{lem:Symmetries} we prove a weaker version of Noether's theorem which has appeared previously using other machinery -- for example, see \cite{Shadwick1980a}.  We end by discussing how generalized symmetries relate to the Jacobi fields of Pinkall and Sterling \cite{Pinkall1989}. 

It is most convenient to study symmetries on $\M{\infty}$.  There we have the following
\begin{defn}
A real vector field $v$ on $\M{\infty}$ is a {\bf generalized symmetry of order $r$} for $(\M{\infty},\mci^{(\infty)})$ if $\mcl_v(\mci^{(l)}) \subset \mci^{(l+r)}$ for all $l \geq 0$.  A {\bf trivial generalized symmetry} for $(\M{\infty},\mci^{(\infty)})$ is a generalized symmetry $v$ that satisfies $v \lhk \I{k}=0$.  
\end{defn}
A natural candidate for a trivial generalized symmetry is $\Re(e_{-1})$ or $\Im(e_{-1})$ since $e_{-1} \lhk \I{\infty}=0$.  We calculate that 
\begin{align*}
 \mcl_{e_{-1}}\eta_l&=\eta_{l+1}\\
\mcl_{e_{-1}}\eta_l&=-\tau^{l-1},
\end{align*}
showing that $e_{-1}$ is an order $1$ generalized symmetry of $(\M{\infty},\mci^{(\infty)})$.  By the observation above it is trivial.  In fact the same is true for $\Re(Q e_{-1})$ for any complex valued function $Q$ on $\M{\infty}$.

\begin{defn}
A {\bf proper} generalized symmetry is a generalized symmetry $v$ also satisfying $v \lhk \zeta=0$.
\end{defn}
In \cite{Bryant2003} the proper generalized symmetries are realized as a quotient of the space of all generalized symmetries.  Using our specified coframe allows us to recognize them as a subspace rather than a quotient. The following lemma has been proven previously in many other contexts.  

\begin{lem}\label{lem:Symmetries}
Let $v$ be a (real) vector field on $\M{\infty}$ such that $\zeta(v)=0$ and let $g=\eta_0(v)$.  Then $v$ is a generalized symmetry of order one of $(\M{\infty},\mci^{(\infty)})$ if and only if $\eta_i(v)=(e_{-1})^{i}(g)$ and $g$ is a solution to Equation \eqref{eq:ED}.
\end{lem}
\begin{proof}
The $\zeta$-coefficient of $\mcl_v(\eta_0)$ is $v^1-e_{-1}(g)$.  Thus the condition that $\mcl_v(\eta_0) \in \I{2}$ implies that $v^1=e_{-1}(g)$.  In general we find that the $\zeta$-coefficient of $\mcl_v(\eta_l)$ is $v^{l+1}-e_{-1}^{l+1}(g)$.  This implies that  $\eta_i(v)=(e_{-1})^{i}(g)$ for all $i \geq 0$. 

Using this, the $\zetab$--coefficient of $\mcl_v(\eta_0)$ is $e_{\ol{-1}}e_{-1}g+f_u g$.  In general, the $\zetab$--coefficient of $\mcl_v(\eta_i)$ is 
\be\nonumber
e_{-1}v^i+\sum_{j=0}^{i-1}{i-1 \choose j}T^{i-1-j}_u v^j.
\ee 
Using $v^j=e_{-1}^{j}(g)$ and $T^i=(e_{-1})^i(f)$ this becomes $(e_{-1})^{i-1}(e_{\ol{-1}}e_{-1}g+f_u g)$.
\end{proof}

To state Noether's theorem in this context we need to recall a standard definition and introduce a refinement of the space of generalized symmetries.
\begin{defn}
If $v$ is a generalized symmetry, then $\eta_0(v)$ is its {\bf generating function}.
\end{defn}
\begin{defn}
Let $\hat \mcs \subset \mcs$ be the subspace of proper generalizaed symmetries $v$ whose generating functions satisfy \eqref{eq:NoMixing}.
\end{defn}
\begin{prop}\us{(}Noether's Theorem\us{)}
There is an isomorphism between $\hat \mcs$ and $\mcc$ given by sending the generating function of a generalized symmetry to the generating function of a conservation law.
\end{prop}
\begin{proof}
This follows immediately from the definitions, Lemma \ref{lem:Symmetries}, and Theorem \ref{thm:MainTheorem}.
\end{proof}

The central equation to solve in order to produce either generalizaed symmetries or conservation laws is \eqref{eq:ED}.  This equation restricts to any integral manifold of $(M,\mci)$ defined by a solution $u(z,\zb)$ of \eqref{eq:fGordon} to be the linearization of \eqref{eq:fGordon}:
\be\label{eq:Linearization}
A_{z \zb}=f_u A.
\ee 
The Jacobi fields studied in \cite{Pinkall1989} are defined to be solutions to \eqref{eq:Linearization} and thus are generating functions for both proper generalizaed symmetries of $(\M{\infty},\mci^{(\infty)})$ and conservation laws.  This explains the appearance of the canonical Jacobi fields of \cite{Pinkall1989} as generating functions for conservation laws in Lemma \ref{lem:PSSolutions}.  The generating functions for conservation laws/generalized symmetries were produced by Olver \cite{Olver1977} using recursion operators.  However, his treatment is not complete because, in the case of the $\sin$-Gordon equation, he doesn't prove that the recursion operator can be applied indefinitely to generate the full infinite sequence.   More rigorous treatments have since been given by Guthrie \cite{Guthrie1994}, Dorfman \cite{Dorfman1993}, and Sanders and Wang \cite{Sanders2001}.  The methods in  \cite{Dorfman1993} and \cite{Sanders2001} are specifically for evolution equations of the form $u_t=K(u)$, where $K$ depends on $u$ and its derivatives with respect to the other independent variables.  The treatment in \cite{Guthrie1994} is more general. One can also obtain the conservation laws for the hyperbolic case with $ f(u)=-\frac{1}{4}\sin(u)$ using the minus one flow in the work of Terng and Uhlenbeck \cite{Terng2000}.  Presumably the conservation laws studied in the present article are equivalent to those derived by Ward \cite{Ward1995}, but this is not clear to us.

\section{Concluding Remarks}\label{sec:Conclusion}
We end this article with a number of observations.  We begin with some issues internal to the theory of characteristic cohomology.  

The spectral sequence machinery used in Section \ref{sec:FirstApproximation} to get a first approximation to the space of conservation laws is extremely useful.  Without it, one has a bewildering freedom in the choice of a representative which will not be easy to deal with.  However, as we found in what is probably the simplest nontrivial class of elliptic equations, the machinery of Section \ref{sec:FirstApproximation} and the calculations of Section \ref{sec:AlgebraicForm} still leave one with some very difficult equations to verify, even once the generating functions are found.  This suggests that the use of the gauge symmetry (Section \ref{sec:Homogeneous} and in particular Equation \eqref{eq:phiexplicit}) in order to produce a direct relationship between solutions to the linearized equation and undifferentiated conservation laws may prove extremely useful, if not essential, in proving the existence of conservation laws for more complicated EDS.  

We have begun exploring this for more complicated systems such as special Lagrangian $3$--folds in $\C^3$, special Legendrian $3$--folds in $\s{7}$, and the EDS for constant mean curvature surfaces in $3$--dimensional space forms.  In each of these cases, a gauge symmetry allows one to find a direct relationship between solutions of the linearized system and undifferentiated conservation laws.  However, one is still left with the formidable challenge of finding solutions to the linearized equation.  For surface geometries, recursion relations for Killing fields appear to be useful, but for higher dimensional submanifold geometries there is no theory of formal/polynomial Killing fields.  It is unclear how to produce canonical solutions to the linearized equation for these higher dimensional systems.  Adapting the theory of recursion operators \cite{Dorfman1993, Guthrie1994} to this context seems essential to developing a complete theory of exterior differential systems with infinitely many higher--order conservation laws.
 
A characteristic property of integrable equations is that they belong to a hierarchy of higher commuting flows \cite{Wilson1979,Terng2000}.  These higher commuting flows can be understood as a canonical sequence of solutions to the linearization of the original equation.  As described by Mukai--Hidano and Ohnita \cite{Mukai2006}, the Killing fields of harmonic (or primitive) map systems are solutions to the linearization of the harmonic map equation.   The framework of characteristic cohomology and the work in \cite{Pinkall1989} suggest that these Killing fields are canonically defined objects on an appropriate jet space that one may restrict to any solution.  It is not clear how to prove this in general  though.  It would be particularly interesting to develop an approach that could work for integral manifolds with any topology.  Integrable systems approaches to harmonic maps with higher genus domains have begun to appear \cite{Mukai2006,Gerding2009}.  Though at present the only approach to higher genus surface geometries (that don't have a Weierstrass representation) that has born fruit has been through gluing constructions using geometric analysis \cite{Kapouleas1990,Haskins2007}. 

Pinkall and Sterling \cite{Pinkall1989} use the canonical Jacobi fields to define a notion of finite type solution.  In the context of harmonic or primitive maps into homogeneous spaces this has been generalized using the notion of formal and polynomial Killing fields \cite{Burstall1993}.  One can use conservation laws to define a notion of finite type solutions which, in the case at hand, recovers the notion defined by Pinkall and Sterling.  We will expand upon this and the relationship between formal/polynomial Killing fields and conservation laws in a forthcoming article. 
\bibliographystyle{amsplain}
\bibliography{bibliography}

\providecommand{\bysame}{\leavevmode\hbox to3em{\hrulefill}\thinspace}
\providecommand{\MR}{\relax\ifhmode\unskip\space\fi MR }
\providecommand{\MRhref}[2]{%
  \href{http://www.ams.org/mathscinet-getitem?mr=#1}{#2}
}
\providecommand{\href}[2]{#2}
\begin{thebibliography}{10}

\bibitem{Bobenko1991}
Alexander~I. Bobenko, \emph{All constant mean curvature tori in {${\bf R}\sp
  3,\;S\sp 3,\;H\sp 3$} in terms of theta-functions}, Math. Ann. \textbf{290}
  (1991), no.~2, 209--245. \MR{1109632 (92h:53072)}

\bibitem{Bolton1995}
John Bolton, Franz Pedit, and Lyndon Woodward, \emph{Minimal surfaces and the
  affine {T}oda field model}, J. Reine Angew. Math. \textbf{459} (1995),
  119--150. \MR{1319519 (96f:58040)}

\bibitem{Bryant1991}
Robert~L. Bryant, Shiing~Shen Chern, Robert~B. Gardner, Hubert~L. Goldschmidt,
  and Phillip~A. Griffiths, \emph{Exterior differential systems}, Mathematical
  Sciences Research Institute Publications, vol.~18, Springer-Verlag, New York,
  1991. \MR{1083148 (92h:58007)}

\bibitem{Bryant1995}
Robert~L. Bryant and Phillip~A. Griffiths, \emph{Characteristic cohomology of
  differential systems. {I}. {G}eneral theory}, J. Amer. Math. Soc. \textbf{8}
  (1995), no.~3, 507--596. \MR{1311820 (96c:58183)}

\bibitem{Bryant1995a}
\bysame, \emph{Characteristic cohomology of differential systems. {II}.
  {C}onservation laws for a class of parabolic equations}, Duke Math. J.
  \textbf{78} (1995), no.~3, 531--676. \MR{1334205 (96d:58158)}

\bibitem{Bryant2003}
Robert~L. Bryant, Phillip~A. Griffiths, and Daniel Grossman, \emph{Exterior
  differential systems and {E}uler-{L}agrange partial differential equations},
  Chicago Lectures in Mathematics, University of Chicago Press, Chicago, IL,
  2003. \MR{1985469 (2004g:58001)}

\bibitem{Bryant1995b}
Robert~L. Bryant, Phillip~A. Griffiths, and Lucas Hsu, \emph{Hyperbolic
  exterior differential systems and their conservation laws. {I}}, Selecta
  Math. (N.S.) \textbf{1} (1995), no.~1, 21--112. \MR{1327228 (97d:58008)}

\bibitem{Bryant1995c}
\bysame, \emph{Hyperbolic exterior differential systems and their conservation
  laws. {II}}, Selecta Math. (N.S.) \textbf{1} (1995), no.~2, 265--323.
  \MR{1354599 (97d:58009)}

\bibitem{Bryant1995d}
\bysame, \emph{Toward a geometry of differential equations}, Geometry,
  topology, \& physics, Conf. Proc. Lecture Notes Geom. Topology, IV, Int.
  Press, Cambridge, MA, 1995, pp.~1--76. \MR{1358612 (97b:58005)}

\bibitem{Burstall1995}
F.~E. Burstall and F.~Pedit, \emph{Dressing orbits of harmonic maps}, Duke
  Math. J. \textbf{80} (1995), no.~2, 353--382.

\bibitem{Burstall1992}
Francis~E. Burstall, \emph{Harmonic maps and soliton theory}, Mat. Contemp.
  \textbf{2} (1992), 1--18, Workshop on the Geometry and Topology of Gauge
  Fields (Campinas, 1991). \MR{1303156 (96g:58042)}

\bibitem{Burstall1993}
Francis~E. Burstall, Dirk Ferus, Franz Pedit, and Ulrich Pinkall,
  \emph{Harmonic tori in symmetric spaces and commuting {H}amiltonian systems
  on loop algebras}, Ann. of Math. (2) \textbf{138} (1993), no.~1, 173--212.
  \MR{1230929 (94m:58057)}

\bibitem{Clelland1997a}
Jeanne~N. Clelland, \emph{Geometry of conservation laws for a class of
  parabolic partial differential equations}, Selecta Math. (N.S.) \textbf{3}
  (1997), no.~1, 1--77. \MR{1454085 (98d:58005)}

\bibitem{Clelland1997}
\bysame, \emph{Geometry of conservation laws for a class of parabolic
  {{P}{DE}}'s. {{I}{I}}. {N}ormal forms for equations with conservation laws},
  Selecta Math. (N.S.) \textbf{3} (1997), no.~4, 497--515. \MR{1613519
  (99d:35074)}

\bibitem{Dai2000}
Bo~Dai and Chuu-Lian Terng, \emph{B\"acklund transformations, {W}ard solitons,
  and unitons}, J. Differential Geom. \textbf{75} (2007), no.~1, 57--108.
  \MR{2282725 (2008h:58027)}

\bibitem{Fox2009}
{Daniel Fox}, \emph{Boundaries of graphs of harmonic functions},
  arXiv:0903.1018 (2009).

\bibitem{Dodd1977}
R.~K. Dodd and R.~K. Bullough, \emph{Polynomial conserved densities for the
  sine-{G}ordon equations}, Proc. Roy. Soc. London Ser. A \textbf{352} (1977),
  no.~1671, 481--503.

\bibitem{Dorfman1993}
Irene Dorfman, \emph{Dirac structures and integrability of nonlinear evolution
  equations}, Nonlinear Science: Theory and Applications, John Wiley \& Sons
  Ltd., Chichester, 1993. \MR{1237398 (94j:58081)}

\bibitem{Ferapontov2006}
Eugene~V. Ferapontov, Karima~R. Khusnutdinova, and Sergey~P. Tsarev, \emph{On a
  class of three-dimensional integrable {L}agrangians}, Comm. Math. Phys.
  \textbf{261} (2006), no.~1, 225--243. \MR{2193210 (2006j:37073)}

\bibitem{Gardner1974}
Clifford~S. Gardner, John~M. Greene, Martin~D. Kruskal, and Robert~M. Miura,
  \emph{Korteweg-de{V}ries equation and generalization. {VI}. {M}ethods for
  exact solution}, Comm. Pure Appl. Math. \textbf{27} (1974), 97--133.
  \MR{0336122 (49 \#898)}

\bibitem{Gerding2009}
Sebastian;~Pedit Gerding, Aaron;~Heller, \emph{G{LOBAL} {ASPECTS} {OF}
  {INTEGRABLE} {SURFACE} {GEOMETRY}}, To appear in the proceedings of
  "Integrable Systems and Quantum Field Theory at Peyresq, Fifth Meeting".

\bibitem{Griffiths1994}
Phillip~A. Griffiths and Joseph~D. Harris, \emph{Principles of algebraic
  geometry}, Wiley Classics Library, John Wiley \& Sons Inc., New York, 1994,
  Reprint of the 1978 original. \MR{1288523 (95d:14001)}

\bibitem{Guthrie1994}
Graeme~A. Guthrie, \emph{Recursion operators and non-local symmetries}, Proc.
  Roy. Soc. London Ser. A \textbf{446} (1994), no.~1926, 107--114. \MR{1287807
  (95i:35017)}

\bibitem{Haskins2007}
Mark Haskins and Nikolaos Kapouleas, \emph{Special {L}agrangian cones with
  higher genus links}, Invent. Math. \textbf{167} (2007), no.~2, 223--294.

\bibitem{Hitchin1990}
Nigel~J. Hitchin, \emph{Harmonic maps from a {$2$}-torus to the {$3$}-sphere},
  J. Differential Geom. \textbf{31} (1990), no.~3, 627--710. \MR{1053342
  (91d:58050)}

\bibitem{Ibragimov1980}
Nail~H. Ibragimov and A.~B. {\v{S}}abat, \emph{Evolution equations with
  nontrivial {L}ie-{B}acklund group}, Functional Analysis and Its Applications
  \textbf{14} (1980), no.~1, 25--36, Translated from Funktsional'nyi Analiz i
  Ego Prilozheniya. \MR{0565093 (82f:58087)}

\bibitem{Ward1995}
Theodora Ioannidou and Richard~S. Ward, \emph{Conserved quantities for
  integrable chiral equations in {$2+1$} dimensions}, Phys. Lett. A
  \textbf{208} (1995), no.~3, 209--213. \MR{1363152 (97b:58069)}

\bibitem{Ivey2003}
Thomas~A. Ivey and Joseph~M. Landsberg, \emph{Cartan for beginners:
  differential geometry via moving frames and exterior differential systems},
  Graduate Studies in Mathematics, vol.~61, American Mathematical Society,
  Providence, RI, 2003. \MR{2003610 (2004g:53002)}

\bibitem{Kapouleas1990}
Nicolaos Kapouleas, \emph{Complete constant mean curvature surfaces in
  {E}uclidean three-space}, Ann. of Math. (2) \textbf{131} (1990), no.~2,
  239--330.

\bibitem{McIntosh2001}
Ian McIntosh, \emph{Harmonic tori and generalised {J}acobi varieties}, Comm.
  Anal. Geom. \textbf{9} (2001), no.~2, 423--449. \MR{1846209 (2002h:53114)}

\bibitem{Mukai2006}
Mariko Mukai-Hidano and Yoshihiro Ohnita, \emph{Gauge-theoretic approach to
  harmonic maps and subspaces in moduli spaces}, Integrable systems, geometry,
  and topology, AMS/IP Stud. Adv. Math., vol.~36, Amer. Math. Soc., Providence,
  RI, 2006, pp.~191--234. \MR{2222516 (2007d:58020)}

\bibitem{Olver1977}
Peter~J. Olver, \emph{Evolution equations possessing infinitely many
  symmetries}, J. Mathematical Phys. \textbf{18} (1977), no.~6, 1212--1215.
  \MR{0521611 (58 \#25341)}

\bibitem{Olver1993}
\bysame, \emph{Applications of {L}ie groups to differential equations}, second
  ed., Graduate Texts in Mathematics, vol. 107, Springer-Verlag, New York,
  1993. \MR{1240056 (94g:58260)}

\bibitem{Pinkall1989}
Ulrich Pinkall and Ivan Sterling, \emph{On the classification of constant mean
  curvature tori}, Ann. of Math. (2) \textbf{130} (1989), no.~2, 407--451.
  \MR{1014929 (91b:53009)}

\bibitem{Sanders2001}
Jan~A. Sanders and Jing~Ping Wang, \emph{Integrable systems and their recursion
  operators}, Proceedings of the {T}hird {W}orld {C}ongress of {N}onlinear
  {A}nalysts, {P}art 8 ({C}atania, 2000), vol.~47, 2001, pp.~5213--5240.
  \MR{1974732 (2004e:37109)}

\bibitem{Shadwick1980}
William~F. Shadwick, \emph{The {H}amilton-{C}artan formalism for higher order
  conserved currents. {I}. {R}egular first order {L}agrangians}, Rep. Math.
  Phys. \textbf{18} (1980), no.~2, 243--256 (1983). \MR{730751 (85b:58050)}

\bibitem{Shadwick1980a}
\bysame, \emph{Noether's theorem and {S}teudel's conserved currents for the
  sine-{G}ordon equation}, Lett. Math. Phys. \textbf{4} (1980), no.~3,
  241--248. \MR{583089 (81j:58049)}

\bibitem{Terng1980}
Chuu-Lian Terng, \emph{A higher dimension generalization of the sine-{G}ordon
  equation and its soliton theory}, Ann. of Math. (2) \textbf{111} (1980),
  no.~3, 491--510. \MR{577134 (82j:58069)}

\bibitem{Terng2000}
Chuu-Lian Terng and Karen Uhlenbeck, \emph{B\"acklund transformations and loop
  group actions}, Comm. Pure Appl. Math. \textbf{53} (2000), no.~1, 1--75.
  \MR{1715533 (2000k:37116)}

\bibitem{Uhlenbeck1989}
Karen Uhlenbeck, \emph{Harmonic maps into {L}ie groups: classical solutions of
  the chiral model}, J. Differential Geom. \textbf{30} (1989), no.~1, 1--50.
  \MR{1001271 (90g:58028)}

\bibitem{Ziber1979}
A.~V. vZiber and Alexei~B. vSabat, \emph{The {K}lein-{G}ordon equation with
  nontrivial group}, Dokl. Akad. Nauk SSSR \textbf{247} (1979), no.~5,
  1103--1107. \MR{550472 (80k:35060)}

\bibitem{Wang2004}
Sung~Ho Wang, \emph{Conservation laws for a class of third order evolutionary
  differential systems}, Trans. Amer. Math. Soc. \textbf{356} (2004), no.~10,
  4055--4073. \MR{2058518 (2005c:35261)}

\bibitem{Wilson1979}
George Wilson, \emph{Commuting flows and conservation laws for {L}ax
  equations}, Math. Proc. Cambridge Philos. Soc. \textbf{86} (1979), no.~1,
  131--143. \MR{530817 (80k:58059)}

\end{thebibliography}
\end{document}